\documentclass[11pt]{amsart}
\usepackage{}
\usepackage{mathrsfs}
\usepackage{amsfonts}

\usepackage{amssymb,amsthm,amsmath,hyperref,txfonts}

\usepackage{graphicx,float}

\numberwithin{equation}{section}

\newtheorem{thm}{Theorem}[section]
\newtheorem{coro}{Corollary}[section]
\newtheorem{lem}{Lemma}[section]
\newtheorem{rem}{Remark}[section]
\newtheorem{prop}{Proposition}[section]
\newtheorem{defn}{Definition}[section]

\newcommand{\beq}{\begin{eqnarray}}
\newcommand{\eeq}{\end{eqnarray}}
\newcommand{\beqno}{\begin{eqnarray*}}
\newcommand{\eeqno}{\end{eqnarray*}}
\newcommand{\be}{\begin{equation}}
\newcommand{\ee}{\end{equation}}
\newcommand{\beno}{\begin{equation*}}
\newcommand{\eeno}{\end{equation*}}

\newcommand\tl{\tilde}

\newcommand\Dl{\Delta}

\newcommand\M{\mathbb{M}}
\newcommand\N{\mathbb{N}}
\newcommand\nb{\nabla}
\newcommand\nn{\nonumber}
\newcommand{\R}{\mathbb{R}}
\newcommand\Z{\mathbb{Z}}

\newcommand\eps{\epsilon}

\newcommand\fr{\frac}
\newcommand\al{\alpha}
\newcommand{\dv}{\mathrm{div}}

\newcommand\lm{\lambda}
\newcommand\Lm{\Lambda}

\newcommand\Om{\Omega}
\newcommand\om{\omega}
\newcommand\Pe{\mathbb{P}}
\newcommand\pr{\partial}
\newcommand{\Rey}{\mathrm{Re}}
\newcommand{\We}{\mathrm{We}}
\newcommand{\ddl}{\dot{\Dl}_q}

\allowdisplaybreaks

\begin{document}
\title[Compressible Oldroyd-B model with non-small
coupling parameter]{Global solution in critical spaces to the compressible Oldroyd-B model with non-small coupling parameter}

\author{Ruizhao Zi}

\address{School of Mathematics and Statistics, Central
China Normal University, Wuhan 430079, P. R. China}
\email{ruizhao3805@163.com}



\begin{abstract}
This paper is dedicated to the global well-posedness issue of the compressible Oldroyd-B model in the whole space $\R^d$ with $d\ge2$.  It is shown that this set of equations admits a unique global solution in a certain critical Besov space provided the initial data, but not necessarily the coupling parameter, is small enough. This result extends the work by Fang and the author [{J. Differential Equations}, {256}(2014), 2559--2602] to the non-small coupling parameter case.
\end{abstract}
\maketitle

\section{Introduction}
\subsection{The modeling}
Unsteady flows of viscoelastic fluids is described by the
conservation of mass, and of momentum, and the constitutive equation
of the fluid. Particularly, the Oldroyd-B fluids in
$(0,T)\times\mathbb{R}^d, T>0, d\ge2$ is governed by
 \beq\label{unsteady viscoelastic}
\begin{cases}
\pr_t\rho+\dv(\rho u)=0,\\
\rho(\pr_tu+(u\cdot\nabla) u)=\dv (\tau-p\, \mathrm{Id}),\\
\tau+\lambda
\frac{\mathcal{D}_{\alpha}\tau}{\mathcal{D}t}=2\eta(D(u)+\mu\frac{\mathcal{D}_{\al}D(u)}{\mathcal{D}t}),\\
( u,\rho, \tau)|_{t=0}=(\rho_0, u_0, \tau_0).
\end{cases}
 \eeq
The unknown $(\rho, u, \tau)$ are the density,  velocity and
symmetric tensor of constrains, respectively. $\mathrm{Id}$ is the identity
tensor, and the smooth function $p=p(\rho)$ is the pressure.
Moreover,  $\eta$ is the total viscosity of the fluid, $\lm>0$ is
the relaxation time, and $\mu$ is the retardation time with
$0<\mu<\lm$. For a tensor $\M$, we denote by
$\frac{\mathcal{D}_{\al}\M}{\mathcal{D}t}$ an objective derivative
of the tensor $\M$, defined by
 \beno
\frac{\mathcal{D}_\al\M}{\mathcal{D}t}=(\pr_t+(u\cdot\nb))\M+\M
W(u)-W(u)\M-\al(D(u)\M+\M D(u)),
 \eeno
where $D(u)=1/2(\nabla u+(\nabla u)^\top)$, $W(u)=1/2(\nabla
u-(\nabla u)^\top)$ are the deformation tensor and the vorticity
tensor respectively, and $\al$ is a parameter in [-1,1].

The symmetric tensor of constrains $\tau$ could be decomposed into
the Newtonian part and the elastic part $\tau_e$, i.e.,
 \be\label{decom_tau}
\tau=2\eta_sD(u)+\tau_e,
 \ee
where $\eta_s=\eta\mu/\lm$  is the solvent viscosity. Substituting
\eqref{decom_tau} into $\eqref{unsteady viscoelastic}_3$, we find
that $\tau_e$ satisfies
 \be\label{tau_e}
\tau_e+\lambda\frac{\mathcal{D}_\al\tau_e}{\mathcal{D}t}=2\eta_e
D(u),
 \ee
where $\eta_e=\eta-\eta_s$ is the polymer  viscosity.

For the sake of simplicity, we denote $\tau_e$ by $\tau$ from now on. Then it follows from
\eqref{unsteady viscoelastic}--\eqref{tau_e} that $(\rho, u,  \tau)$ solves
\begin{eqnarray}\label{COB4}
\begin{cases}
\pr_t\rho+\dv(\rho u)=0,\\
\rho(\pr_tu+(u\cdot\nabla) u)-\eta_s(\Delta u+\nabla\dv u)+\nabla
p=\dv\tau,\\
\lambda(\pr_t\tau+(u\cdot\nabla)\tau+g_\al(\tau,
\nabla u))+\tau=2\eta_e D(u), \\
(u,\rho, \tau)|_{t=0}=(u_0,\rho_0,\tau_0),
\end{cases}
\end{eqnarray}
where $g_\al(\tau, \nabla u):=\tau
W(u)-W(u)\tau-\al\left(D(u)\tau+\tau D(u)\right)$ with $D(u))$,
$W(u)$ and $\al$ defined as before. For more explanations on the
modeling, see \cite{Oldroyd58,Talhouk94} and references therein.

Introducing some dimensionless variables, c.f. \cite{FZ14}, then  system \eqref{COB4} can be reduced to
\begin{eqnarray}\label{COB1dimensionless}
\begin{cases}
\displaystyle \pr_tb+\frac{\mathrm{Re}\ \dv u}{\eps}+\dv (b u)=0,\\
\displaystyle\mathrm{Re}\left(\pr_tu+(u\cdot\nb) u\right)+\frac{\nb b}{\eps}-(1-\om)(\Dl +\nb\dv) u\\
\displaystyle\quad\quad\quad\quad\quad\quad\quad\quad=-\frac{(1-\om)\eps b}{\mathrm{Re}+\eps b}(\Dl +\nb\dv)u+K^\eps(\eps b)\frac{\nb b}{\eps}+\frac{\Rey\dv\tau}{\mathrm{Re}+\eps b},\\
\displaystyle\We(\pr_t\tau+(u\cdot\nabla)\tau+g_\al(\tau,
\nabla u))+\tau=2\om D(u),
\end{cases}
\end{eqnarray}
where the parameters $\eps, \Rey, \We, \om$ are Mach number,  Reynolds number, Weissenberg number and coupling constant, respectively, with $\om\in(0,1)$. Moreover,
 \be\label{K^eps}
K^\eps(c)=\frac{c}{\mathrm{Re}+c}-\eps^2\frac{\mathrm{Re}\left(\frac{dp}{d\rho}(\mathrm{Re}+c)-\frac{dp}{d\rho}(\mathrm{Re})\right)}
{\mathrm{Re}+c}.
 \ee
We would like to point out that, letting $\eps\rightarrow0$ in \eqref{COB1dimensionless}, then we obtain (see \cite{FZ14}) the following incompressible  Oldroyd-B model in dimensionless variables:
\begin{eqnarray}\label{IOBdimensionless}
\begin{cases}
\mathrm{Re}\left(u_t+(u\cdot\nabla) u\right)-(1-\om)\Delta u+\nabla
\Pi=\dv\tau,\\
\mathrm{We}(\tau_t+(u\cdot\nabla)\tau+g_\al(\tau,
\nabla u))+\tau=2\om D(u), \\
\dv u=0,
\end{cases}
\end{eqnarray}
where $\Pi$ is the pressure which is the Lagrange multiplier for the divergence free condition.

In this paper, we focus on the case $\eps=1$. Setting $a:=\frac{b}{\Rey}$, then $(a, u,\tau)$ takes the form
\begin{eqnarray}\label{COB2dimensionless}
\begin{cases}
\displaystyle \pr_ta+ \dv u+\dv (a u)=0,\\
\displaystyle \pr_tu+(u\cdot\nb) u-\frac{1}{\Rey}\mathcal{A} u+\nb a-\frac{1}{\Rey}\dv\tau \\
\displaystyle\quad\quad\quad\quad\quad\quad\quad\quad=-\frac{1}{\Rey}I(a)\left(\mathcal{A}u+\dv\tau \right)+\tl{K}(a)\nb a,\\
\displaystyle\pr_t\tau+(u\cdot\nabla)\tau+g_\al(\tau,
\nabla u)+\fr{1}{\We}\tau=\fr{2\om}{\We} D(u),
\end{cases}
\end{eqnarray}
where $\mathcal{A}:=(1-\om)(\Dl+\nb\dv)$, $I(a):=\frac{a}{1+a}$ and
\beno
\tl{K}(a):=\frac{a}{1+a}-\frac{\frac{dp}{d\rho}(\mathrm{Re}(1+a))-\frac{dp}{d\rho}(\mathrm{Re})}
{1+a}.
\eeno

The theory of Oldroyd-B fluids recently gained quite some attention.
Most of the results on Oldroyd-B fluids in the literature are about the incompressible model.
The study  of the incompressible Oldroyd-B model started by a pioneering
paper given by Guillop\'e and Saut \cite{GS90}. They proved that (i) system \eqref{IOBdimensionless} admits a unique local strong solution in
suitable Sobolev spaces $H^s(\Omega)$ for  {\em bounded domains} $\Omega \subset \R^3$; and (ii) this solution
is global provided the data as well as the coupling
constant $\om$ between the velocity $u$ and the symmetric tensor of the constrains $\tau$  are sufficiently small.
For extensions to this results to the $L^p$-setting, see the work of
Fernand\'ez-Cara, Guill\'en and Ortega \cite{FGG98}. Later on, Molinet and Talhouk \cite{MT04} removed
the smallness restriction on the coupling constant $\om$ in
\cite{GS90}. The situation of {\em exterior domains} was considered
first in \cite{HNS12}, where the existence of a unique global strong
solution defined in certain function spaces was proved provided the
initial data and the coupling parameter $\om$ are small enough. Recently, Fang, Hieber and the author \cite{FHZ13} improved
the main result given in \cite{HNS12} to the situation of non-small coupling constant.

 For the {\em scaling invariant} approach, and $\Om=\R^d, d\ge2$, Chemin and Masmodi in \cite{CM01} proved the existence and uniqueness of the
 global solution to the Oldroyd-B model \eqref{IOBdimensionless} with initial data $(u_0,\tau_0)$ belonging to the critical space
 $\left(\dot{B}^{\fr{d}{p}-1}_{p,1}\right)^d\times \left(\dot{B}^{\fr{d}{p}}_{p,1}\right)^{d\times d}$ for any $p\in[1,\infty)$.
A smallness restriction on the coupling constant $\om$ is needed in this result.  Afterwards, for general $\om\in(0,1)$,
 Chen and Miao \cite{Chen-Miao08} constructed global solutions to  the incompressible Oldroyd-B model with small initial data in
$B^s_{2,\infty}, s>\frac{d}{2}$. For the critical $L^p$  framework, Fang, Zhang and the author \cite{Zi-Fang-Zhang14} removed the smallness restriction on $\om$ in \cite{CM01} very recently.

For the {\em weak solutions} of incompressible Oldroyd-B fluids,  see the work of   Lions and Masmoudi \cite{LM00} for the case $\al=0$. The general case $\al\neq0$ is still open up to now. As for  the {\em blow-up criterions} of  the incompressible Oldroyd-B model,
there are works  \cite{CM01,KMT08,LMZ10}.  Besides, we would like to mention that
 Constantin and Kliegl \cite{CK12} proved the global regularity of solutions in two dimensional case for the Oldroyd-B fluids with {\em diffusive stress}.    An approach based on the {\em deformation tensor} can be found in \cite{Lei10,Lei07,LLZ08,LZ05,LLZ05,LZ08,Qian10,Zhang-Fang12}.

On the other hand, the studies on compressible Oldroyd-B model have thrown up some interesting results.
 Lei \cite{Lei06} and Gullop\'e,
Salloum and Talhouk \cite{GST10} investigated the
incompressible limit problem of the compressible Oldroyd-B model  in
a {\em torus} and bounded domain $\Omega \subset \R^3$, respectively. They showed that the  compressible flows with {\em
well-prepared} initial data converge to incompressible ones when the
Mach number $\eps$ converges to zero. The case of  {\em ill prepared} initial data was studied by Fang and the author \cite{FZ14} in the whole space $\R^d, d\ge2$. In particular, if $\eps=1$, we also obtained in \cite{FZ14} the  existence and uniqueness of the global  solution in critical spaces to system \eqref{COB1dimensionless} with small coupling constant $\om$.
The unique  local   strong solution to \eqref{COB4}  with initial density $\rho_0$ {\em vanishing
from below} and a blow-up criterion for this soltion were established in  \cite{Fang-Zi13}. For the compressible Oldroyd type
model based on the {\em deformation tensor}, see the results
\cite{DLZ12, HL14, HW11, Qian-Zhang10} and references therein.

The aim of this paper is to study the compressible Oldroyd-B model \eqref{COB2dimensionless} in the critical framework. This approach goes back to the pioneering work by Fujita and Kato \cite{FK64} for the classical incompressible Navier-Stokes equations.   We refer to \cite{C97,CMP93,KT01,W80} for a recent panorama. Strictly speaking, the compressible Oldroyd-B model does not have any scaling invariance. However, if we neglect the coupling term $\dv \tau$ and the damping term $\tau$, it is found that \eqref{COB4} is invariant under the transformation
\begin{gather*}
(\rho_0, u_0, \tau_0)\rightarrow (\rho_0(\ell x),\ell u_0(\ell x),\tau_0(\ell x)),\\
( \rho(t,x), u(t,x), \tau(t,x), p(t,x))\rightarrow (\rho(\ell^2t, \ell x), \ell u(\ell^2t, \ell x), \tau(\ell^2t, \ell x), \ell^2p(\ell^2t,\ell x)),
\end{gather*}
for any $\ell>0$. This motivates us to consider system \eqref{COB2dimensionless} with initial data $(a_0, u_0,\tau_0)\in \dot{B}^{\fr{d}{2}-1,\fr{d}{2}}_{2,1}\times  \left(\dot{B}^{\fr{d}{2}-1}_{2,1}\right)^d\times \left(\dot{B}^{\fr{d}{2}}_{2,1}\right)^{d\times d}$. Different from our previous results in \cite{FZ14}, the coupling constant $\om$ we investigate here is not small any more and thus the problem is much more complicated. As a matter of fact, in \cite{FZ14}  $(a, u)$ and $\tau$ are treated separately. To be more precise,  we bound $(a, u)$ by using the estimates obtained by Danchin \cite{Danchin00} for the linearized system of barotropic compressible Naviter-Stokes equations, namely
\beq\label{LNS}
\begin{cases}
\begin{array}{rrl}
a_t+\Lambda d&=&0,\\
d_t-\Dl d-\Lambda a&=&0,
\end{array}
\end{cases}
\eeq
where $\Lambda:=(-\Dl)^{\fr12}$. The linear coupling term $\dv\tau$ in the momentum equation is regarded as a source term, and the symmetric tensor of constrains $\tau$ is bounded with the aid of the well known estimates for transport equation. This is an easy way to get the global estimates of $(a, u, \tau)$ since the coupling between $a$ and $\tau$ is neglected, nevertheless, in order to close the estimates, the price we have to pay is to impose some smallness restriction on the coupling constant $\om$. For general $\om\in(0,1)$, we must consider fully the coupling between $a, u$ and $\tau$, and deal with $(a, u,\tau)$ as a whole.

Let us now explain the main ingredients of the proof. Motivated by our previous result for incompressible Oldroyd-B model \cite{Zi-Fang-Zhang14}, we first consider the system of $(a, u, \dv\tau)$. Indeed, in view of the scaling above, $u$ possesses the same regularity with $\dv\tau$ instead of $\tau$, that is why it is more convenient to treat $(a, u,\dv\tau)$ as a whole in the process of energy estimates in Besov spaces. To do so, our proof relies heavily on  the following decomposition on $u$ and $\dv\tau$:
 \beno
 u=\Pe u+\Pe^\bot u, \quad\mathrm{and}\quad \dv\tau=\Pe \dv\tau+\Pe^\bot\dv\tau,
 \eeno
where $\Pe:=\mathrm{Id}+\nb(-\Dl)^{-1}\dv$ is the Leray operator, and $\Pe^\bot:=-\nb(-\Dl)^{-1}\dv$. Applying $\Pe^\bot$ and $\Pe^\bot\dv$ to the second and third equation of \eqref{COB2dimensionless} respectively, we obtain
\begin{eqnarray}\label{COB-c}
\begin{cases}
\displaystyle \pr_ta+ \dv u+\dv (a u)=0,\\
\displaystyle \pr_t\Pe^\bot u+\Pe^\bot((u\cdot\nb)  u)-\frac{2(1-\om)}{\Rey}\Dl \Pe^\bot u+\nb a-\frac{1}{\Rey}\Pe^\bot\dv\tau \\
\displaystyle\quad\quad\quad\quad\quad\quad\quad\quad=\Pe^{\bot}\left(-\frac{1}{\Rey}I(a)\left(\mathcal{A}u+\dv\tau \right)+\tl{K}(a)\nb a\right),\\
\displaystyle\pr_t\Pe^\bot\dv\tau+\fr{1}{\We}\Pe^\bot\dv\tau-\fr{2\om}{\We} \Dl\Pe^\bot u=-\Pe^\bot\dv\left((u\cdot\nb) \tau+ g_\al(\tau,
\nabla u)\right),
\end{cases}
\end{eqnarray}
and
\begin{eqnarray}\label{COB-i}
\begin{cases}
\displaystyle \pr_t\Pe u+\Pe\left((u\cdot\nb) u\right)-\frac{1-\om}{\Rey}\Dl\Pe u-\frac{1}{\Rey}\Pe\dv\tau =-\frac{1}{\Rey}\Pe\left(I(a)\left(\mathcal{A}u+\dv\tau \right)\right),\\
\displaystyle\pr_t\Pe\dv\tau+\fr{1}{\We}\Pe\dv\tau-\fr{\om}{\We} \Dl\Pe u=-\Pe\dv\left((u\cdot\nb) \tau+ g_\al(\tau,
\nabla u)\right).
\end{cases}
\end{eqnarray}
Obviously, the linear part of \eqref{COB-i} is the same with that of the auxiliary system of $(u, \Pe\dv\tau)$ for the incompressible Oldroyd-B model (see \cite{Zi-Fang-Zhang14} (1.12)), so the key point of this paper is to deal with the so called compressible  part, i. e., system \eqref{COB-c}. For the same reason as the case of  barotropic Navier-Stokes equations \cite{Danchin00}, we have to study the high frequency and low frequency part of system \eqref{COB-c} in different ways. Roughly speaking, it is necessary to bound
\be\label{kq}
k_q+\min(2^{2q}, 1)\int_0^tk_qdt',
\ee
where
\beqno
k_q=\begin{cases}
\|a_q\|_{L^2}+\|\Pe^\bot u_q\|_{L^2}+\|\Pe^\bot\dv\tau_q\|_{L^2}, \quad\quad\mathrm{if}\quad q\le q_0,\\
\|\nb a_q\|_{L^2}+\|\Pe^\bot u_q\|_{L^2}+\|\Pe^\bot\dv\tau_q\|_{L^2}, \quad\,\,\mathrm{if}\quad q>q_0,
       \end{cases}
\eeqno
for some $q_0\in\Z$. In order to get the decay of $a$ and $\dv\tau$, we make full use of the linear coupling terms $\nb a$ and $\dv \tau$, and the estimates are very delicate both in low and high frequency cases.  Furthermore, different from the barotropic Navier-Stokes equations,  it is shown that there is a gap between the high frequency and low frequency  estimates of $k_q$. In other words, we can estimate \eqref{kq} for $q> q_0$ and $q\le q_1$ with $q_1<q_0$. To overcome this difficulty, we introduce a new quantity
    \beno
    \tl{k}_q:=\|a_q\|_{L^2}+\|\Pe^\bot u_q\|_{L^2}+\|\Lambda^{-1}\Pe^\bot\dv\tau_q\|_{L^2}, \quad\mathrm{if}\quad q_1< q\le q_0.
    \eeno
According to Bernstein's inequality, $ \tl{k}_q$ is equivalent to $k_q$ if $q_1< q\le q_0$. Therefor, it suffices to bound \eqref{kq} with  $k_q$ replaced by $\tl{k}_q$ for $q_1< q\le q_0$. This is the main novel part of this paper. Once the estimates of the compressible part $(a, \Pe^\bot u, \Pe^\bot\dv\tau)$  of $(a, u, \dv\tau)$ is obtained, we can bound the incompressible part $(\Pe u, \Pe\dv\tau)$ in a similar and easier way. Putting them together, we get the estimates of $(a, u, \dv\tau)$. On this basis, we bound $\tau$ directly via the third equation of \eqref{COB2dimensionless}, and hence obtain the global estimates of $(a, u, \tau)$ in the end. More details can be found in Section 3.

We shall obtain the existence and uniqueness of a solution $(a, u,\tau)$ to \eqref{COB2dimensionless} in the following space.

\begin{def}\label{space}
For $T>0$ and $s\in\R$, let us denote
\beno
\mathcal{E}^s_T:=\tl{C}_T(\dot{B}^{s-1,s}_{2,1})\cap L^1_T(\dot{B}^{s+1,s}_{2,1})\times\left(\tl{C}_T(\dot{B}^{s-1}_{2,1})\cap L^1_T(\dot{B}^{s+1}_{2,1})\right)^d\times\left(\tl{C}_T(\dot{B}^{s}_{2,1})\cap L^1_T(\dot{B}^{s}_{2,1})\right)^{d\times d}.
\eeno
We use the notation $\mathcal{E}^s$ if $T=\infty$, changing $[0, T]$ into $[0,\infty)$ in the definition above.
\end{def}

Our main result reads as follows:
\begin{thm}\label{thm-g}
Let $d\ge2$. Assume that $(a_0, u_0,\tau_0)\in \dot{B}^{\fr{d}{2}-1,\fr{d}{2}}_{2,1}\times  \left(\dot{B}^{\fr{d}{2}-1}_{2,1}\right)^d\times \left(\dot{B}^{\fr{d}{2}}_{2,1}\right)^{d\times d}$. There exist two positive constants $c$ and $M$, depending on $d, \om, \Rey$ and $\We$, such that if
 \beno
 \|a_0\|_{\dot{B}^{\fr{d}{2}-1,\fr{d}{2}}_{2,1}} +\|u_0\|_{\dot{B}^{\fr{d}{2}-1}_{p,1}}+\|\tau_0\|_{\dot{B}^{\fr{d}{2}}_{2,1}}\le c,
  \eeno
  then system \eqref{COB2dimensionless} admits a unique global solution $(a, u,\tau)$ in $\mathcal {E}^{\fr{d}{2}}$ with
  \beno
  \|(a, u,\tau)\|_{\mathcal{E}^{\fr{d}{2}}}\leq M\left( \|a_0\|_{\dot{B}^{\fr{d}{2}-1}_{2,1}\cap\dot{B}^{\fr{d}{2}}_{2,1}} +\|u_0\|_{\dot{B}^{\fr{d}{2}-1}_{2,1}}+\|\tau_0\|_{\dot{B}^{\fr{d}{2}}_{2,1}}\right).
  \eeno
\end{thm}

\begin{rem}
In a forthcoming paper, we will deal with the general case $p\neq2$.
\end{rem}

\begin{rem}
For incompressible Oldroyd-B model, using the estimates of the incompressible part in Section 3, we can give a new proof of the result in \cite{Zi-Fang-Zhang14} for $p=2$ without resorting to the Green matrix of the corresponding linearized system.
\end{rem}

The rest part of this paper is organized as follows. In Section 2, we introduce the tools ( the Littlewood-Paley decomposition
and paradifferertial calculus) and give some nonlinear estimates in Besov space. Section 3 is devoted to the global estimates of the paralinearized system \eqref{piOB} of \eqref{COB2dimensionless}. The proof of Theorem \ref{thm-g} is given in Section 4.

\bigbreak\noindent{\bf Notation.}\\
(1) For $a, b\in L^2$, $(a|b)$ denotes the $L^2$ inner product of $a$ and $b$.\\
(2) For $f\in \mathcal{S}'$, $\hat{f}=\mathcal{F}(f)$ is the Fourier transform of $f$; $\check{f}=\mathcal{F}^{-1}(f)$ denotes the inverse Fourier transform of $f$.\\

\section{The Functional Tool Box}
\noindent The results of the present paper rely on the use of  a
dyadic partition of unity with respect to the Fourier variables, the so-called the
\textit{Littlewood-Paley  decomposition}. Let us briefly explain how
it may be built in the case $x\in \R^d$ which the readers may see more details
in \cite{Bahouri-Chemin-Danchin11,Ch1}. Let $(\chi, \varphi)$ be a couple of $C^\infty$ functions satisfying
$$\hbox{Supp}\chi\subset\{|\xi|\leq\frac{4}{3}\},
\ \ \ \
\hbox{Supp}\varphi\subset\{\frac{3}{4}\leq|\xi|\leq\frac{8}{3}\},
$$
and
$$\chi(\xi)+\sum_{q\geq0}\varphi(2^{-q}\xi)=1,$$

$$\sum_{q\in \mathbb{Z}}\varphi(2^{-q}\xi)=1, \quad \textrm{for} \quad \xi\neq0.$$
Set $\varphi_q(\xi)=\varphi(2^{-q}\xi),$
$h_q=\mathcal{F}^{-1}(\varphi_q),$ and
$\tilde{h}=\mathcal{F}^{-1}(\chi)$. The dyadic blocks and the low-frequency cutoff operators are defined for all $q\in\mathbb{Z}$ by
$$\dot{\Delta}_{q}u=\varphi(2^{-q}\mathrm{D})u=\int_{\R^d}h_q(y)u(x-y)dy,$$
$$\dot{S}_qu=\chi(2^{-q}\mathrm{D})u=\int_{\R^d}\tl{h}_q(y)u(x-y)dy.$$
Then
\begin{equation}\label{e2.1}
u=\sum_{q\in \mathbb{Z}}\Delta_qu,
\end{equation}
holds for tempered distributions {\em modulo polynomials}. As working modulo polynomials is not appropriate for nonlinear problems, we
shall restrict our attention to the set $\mathcal {S}'_h$ of tempered distributions $u$ such that
$$
\lim_{q\rightarrow-\infty}\|\dot{S}_qu\|_{L^\infty}=0.
$$
Note that \eqref{e2.1} holds true whenever $u$ is in $\mathcal{S}'_h$ and that one may write
$$
\dot{S}_qu=\sum_{p\leq q-1}\dot{\Dl}_{p}u.
$$
Besides, we would like to mention that the Littlewood-Paley decomposition
has a nice property of quasi-orthogonality:
\begin{equation}\label{e2.2}
\dot{\Delta}_p\dot{\Delta}_qu\equiv 0\ \ \hbox{if}\ \ \ |p-q|\geq 2\ \
\hbox{and}\ \ \dot{\Delta}_p(\dot{S}_{q-1}u\dot{\Delta}_qu)\equiv 0\ \ \hbox{if}\ \ \
|p-q|\geq 5.
\end{equation}
One can now  give the definition of
homogeneous Besov spaces.
\begin{defn}\label{D2.1}
For $s\in\R$, $(p,r)\in[1,\infty]^2$, and
$u\in\mathcal{S}'(\R^d),$ we set
$$\|u\|_{\dot{B}_{p,r}^s}=\left\|2^{ sq}\|\dot{\Delta}_qu\|_{L^p} \right\|_{\ell^r}.$$
We then define the space
$\dot{B}_{p,r}^s:=\{u\in\mathcal{S}'_h(\R^d),\
\|u\|_{\dot{B}_{p,r}^s}<\infty\}$.
\end{defn}
Since homogeneous Besov spaces fail to have nice inclusion properties, it is wise to define {\em hybrid Besov spaces} where the growth conditions satisfied by the dyadic blocks are different for low and high frequencies. In fact,  hybrid Besov spaces played a crucial role for proving global well-posedness of compressible barotropic Navier-Stokes equations in critical spaces \cite{Danchin00}. Let us now define the hybrid Besov spaces that we need. Here our notations are somehow different from those in \cite{Danchin00}.
\begin{defn}\label{def-hybrid}
Let $s, t\in\R$, and
$u\in\mathcal{S}'(\R^d)$. For some fixed $q_0\in\Z$, we set
\beno
\|u\|_{\dot{B}^{s,t}_{2,1}}:=\sum_{q\le q_0}2^{qs}\|\ddl u\|_{L^2}+\sum_{q>q_0}2^{qt}\|\ddl u\|_{L^2}.
\eeno
We then define the space
$\dot{B}_{2,1}^{s,t}:=\{u\in\mathcal{S}'_h(\R^d),\
\|u\|_{\dot{B}_{2,1}^{s,t}}<\infty\}$.
\end{defn}
\begin{rem}
For all $s, t\in\R, \tl{q}_0\in\Z$, and
$u\in\mathcal{S}'(\R^d)$, setting
\beno
\|u\|_{\tl{B}^{s,t}_{2,1}}:=\sum_{q\le \tl{q}_0}2^{qs}\|\ddl u\|_{L^2}+\sum_{q>\tl{q}_0}2^{qt}\|\ddl u\|_{L^2},
\eeno
then it is easy to verify that $\|u\|_{\tl{B}^{s,t}_{2,1}}\approx\|u\|_{\dot{B}^{s,t}_{2,1}}$.
\end{rem}

\bigbreak\noindent{\bf Notation.} We will use the following notation:\\
\beno
\|u^l\|_{\dot{B}^s_{2,1}}:=\sum_{q\le q_0}2^{qs}\|\ddl u\|_{L^2},\quad \mathrm{and} \quad \|u^h\|_{\dot{B}^t_{2,1}}:=\sum_{q> q_0}2^{qt}\|\ddl u\|_{L^2}.
\eeno
Obviously,
\beno
\|u\|_{\dot{B}_{2,1}^{s,t}}=\|u^l\|_{\dot{B}^s_{2,1}}+\|u^h\|_{\dot{B}^t_{2,1}}.
\eeno

The following lemma describes the way derivatives act on spectrally localized functions.
\begin{lem}[Bernstein's inequalities]\label{Bernstein}
Let $k\in\N$ and $0<r<R$. There exists a constant $C$ depending on $r, R$ and $d$ such that for all $(a,b)\in[1,\infty]^2$, we have for all $\lm>0$ and multi-index $\al$
\begin{itemize}
\item If $\mathrm{Supp} \hat{f}\subset B(0,\lm R)$, then $\sup_{\al=k}\|\pr^\al f\|_{L^b}\le C^{k+1}\lm^{k+d(\fr1a-\fr1b)}\|f\|_{L^a}$.
\item If $\mathrm{Supp} \hat{f}\subset \mathcal{C}(0,\lm r, \lm R)$, then $C^{-k-1}\lm^k\|f\|_{L^a}\le\sup_{|\al|=k}\|\pr^\al f\|_{L^a}\le C^{k+1}\lm^k\|f\|_{L^a}$
\end{itemize}
\end{lem}

Let us now state some classical
properties for the Besov spaces.
\begin{prop}\label{prop-classical}
For all $s, s_1, s_2\in\R$, $1\le p, p_1, p_2, r, r_1, r_2\le\infty$, the following properties hold true:

\begin{itemize}
\item If $p_1\leq p_2 $  and $r_1\leq r_2,$ then
$\dot{B}_{p_1,r_1}^{s}\hookrightarrow
\dot{B}_{p_2,r_2}^{s-\frac{d}{p_1}+\frac{d}{p_2}}$.

\item If $s_1\neq s_2$ and $\theta\in(0,1)$,
$\left[\dot{B}_{p,r_1}^{s_1},\dot{B}_{p,r_2}^{s_2}\right]_{(\theta,r)}=\dot{B}_{p,r}^{\theta
s_1+(1-\theta)s_2}.$

\item For any smooth homogeneous of degree $m\in\Z$ function $F$ on $\R^d\backslash\{0\}$, the operator $F(D)$ maps $\dot{B}^{s}_{p,r}$ in $\dot{B}^{s-m}_{p,r}$.
\end{itemize}
\end{prop}

Next we  recall a few nonlinear estimates in Besov spaces which may be
obtained by means of paradifferential calculus. Firstly introduced
 by J. M. Bony in \cite{Bony81}, the paraproduct between $f$
and $g$ is defined by
$$\dot{T}_fg=\sum_{q\in\mathbb{Z}}\dot{S}_{q-1}f\dot{\Delta}_qg,$$
and the remainder is given by
$$\dot{R}(f,g)=\sum_{q\geq -1}\dot{\Delta}_qf\tilde{\dot{\Delta}}_qg$$
with
$$\tilde{\dot{\Delta}}_qg=(\dot{\Delta}_{q-1}+\dot{\Delta}_{q}+\dot{\Delta}_{q+1})g.$$
We have the following so-called Bony's decomposition:
 \be\label{Bony-decom}
fg=\dot{T}_fg+\dot{T}_gf+\dot{R}(f,g).
 \ee
The paraproduct $\dot{T}$ and the remainder $\dot{R}$ operators satisfy the following
continuous properties.

\begin{prop}\label{p-TR}
For all $s\in\mathbb{R}$, $\sigma>0$, and $1\leq p, p_1, p_2, r, r_1, r_2\leq\infty,$ the
paraproduct $\dot T$ is a bilinear, continuous operator from $L^{\infty}\times \dot{B}_{p,r}^s$ to $
\dot{B}_{p,r}^{s}$ and from $\dot{B}_{\infty,r_1}^{-\sigma}\times \dot{B}_{p,r_2}^s$ to
$\dot{B}_{p,r}^{s-\sigma}$  with $\frac{1}{r}=\min\{1, \frac{1}{r_1}+\frac{1}{r_2}\}$. The remainder $\dot R$ is bilinear continuous from
$\dot{B}_{p_1,r_1}^{s_1}\times \dot{B}_{p_2,r_2}^{s_2}$ to $
\dot{B}_{p,r}^{s_1+s_2}$ with
$s_1+s_2>0$,  $\frac{1}{p}=\frac{1}{p_1}+\frac{1}{p_2}\leq1$, and $\frac{1}{r}=\frac{1}{r_1}+\frac{1}{r_2}\leq1$.
\end{prop}
In view of \eqref{Bony-decom}, Proposition \ref{p-TR} and Bernstein's inequalities,  one easily deduces the following  product estimates:
\begin{coro}\label{coro-product}
Let $p\in[1,\infty]$. If $s_1, s_2\le \fr{d}{p}$ and $s_1+s_2>d\max \{0,\fr{2}{p}-1\}$, then there holds
\be\label{product1}
\|uv\|_{\dot{B}^{s_1+s_2-\frac{d}{p}}_{p,1}}\leq C\|u\|_{\dot{B}^{s_1}_{p,1}}\|v\|_{\dot{B}^{s_2}_{p,1}}.
\ee
\end{coro}
In the following, we shall give a commutator estimate, which will be used to deal with the convection terms.
\begin{lem}\label{lem-commu}
Let $\Lm:=\sqrt{-\Dl}$, then
\beq\label{commu}
\|[\Lm^{-1}, \dot{S}_{q-1}v\cdot\nb]\ddl u\|_{L^2}\le C\|\nb \dot{S}_{q-1}v\|_{L^\infty}\|\Lm^{-1}\ddl u\|_{L^2}.
\eeq
\end{lem}
\begin{proof}
Noting first that the support of $\dot{S}_{q-1}v\cdot\nb \ddl u$ lies in the annulus $\tl{\mathcal{C}}:=B(0,\frac{2}{3})+\{\frac{3}{4}\leq|\xi|\leq\frac{8}{3}\}$ , we
choose a smooth function $\tl{\varphi}$ supported in an annulus and
with value 1 on a neighborhood of $\tl{\mathcal{C}}$. Denote $\bar{\dot{\Dl}}_{q}=\tl{\varphi}(2^{-q}\mathrm{D})$,
i.e.,
$(\bar{\dot{\Dl}}_{q}\phi)^{\hat{}}=\tl{\varphi}(2^{-q}\xi)\hat{\phi}(\xi)$,
for any $\phi\in \mathcal{S}$. Direct calculations yield
$\Lm^{-1}\bar{\dot{\Dl}}_{q}=2^{-q}(|\cdot|^{-1}\tl{\varphi})(2^{-q}\mathrm{D})$.
Thus
 \beqno
[\Lm^{-1}, \dot{S}_{q-1}v\cdot\nb]\ddl u&=&[\Lm^{-1}\bar{\dot{\Dl}}_{q}, \dot{S}_{q-1}v\cdot\nb]\ddl u\\
&=&[2^{-q}(|\cdot|^{-1}\tl{\varphi})(2^{-q}\mathrm{D}), \dot{S}_{q-1}v\cdot\nb]\ddl u\\
&=&2^{-q}\sum_{1\le k\le d}\int_{\mathbb{R}^d}2^{qd}\mathcal{F}^{-1}(|\cdot|^{-1}\tl{\varphi})(2^{q}y)\\
&&\times\left(\dot{S}_{q-1} v^k(x-y)-\dot{S}_{q-1}v^k(x)\right)\pr_k\ddl u(x-y)dy\\
&=&-2^{-2q}\sum_{1\le k\le d}\int_0^1\int_{\mathbb{R}^d}2^{qd}\mathcal{F}^{-1}(|\cdot|^{-1}\tl{\varphi})(2^{q}y)\\
&&\times(2^{q}y)\cdot \nb\dot{S}_{q-1} v^k(x-ty)\pr_k \ddl u(x-y)dydt.
 \eeqno
Consequently, using convolution inequality, we infer that
 \beqno
&&\|[\Lm^{-1}, \dot{S}_{q-1}v\cdot\nb]\ddl u\|_{L^2}\\
&\leq&C2^{-2q}\|\nb\dot{S}_{q-1}
v\|_{L^\infty}\|\nb\ddl u\|_{L^2}\|y\mathcal{F}^{-1}(|\cdot|^{-1}\tl{\varphi})\|_{L^1}\\
&\le&C\|\nb
\dot{S}_{q-1}v\|_{L^\infty}\|\Lm^{-1}\ddl u\|_{L^2},
 \eeqno
where we have used the fact that $\tilde{\varphi}$ is a smooth
function supported in an annulus, so is
$|\xi|^{-1}\tilde{\varphi}(\xi)$, and hence
$y\mathcal{F}^{-1}(|\cdot|^{-1}\tl{\varphi})$ is integrable. This completes the proof of Lemma \ref{lem-commu}.
\end{proof}

The study of non-stationary PDEs requires spaces of the type
$L^\rho_T(X)=L^\rho(0,T;X)$ for appropriate Banach spaces $X$. In
our case, we expect $X$ to be a  Besov space, so that it
is natural to localize the equations through Littlewood-Paley
decomposition. We then get estimates for each dyadic block and
perform integration in time. But, in doing so, we obtain the bounds
in spaces which are not of the type $L^\rho(0,T;\dot{B}^s_{p,r})$. That
 naturally leads to the following definition introduced by Chemin and Lerner in \cite{CL}.
\begin{defn}\label{defn-chemin-lerne}
For $\rho\in[1,+\infty]$, $s\in\R$, and $T\in(0,+\infty)$, we set
$$\|u\|_{\tilde{L}^\rho_T(\dot{B}^s_{p,r})}=\left\|2^{qs}
\|\dot{\Delta}_qu(t)\|_{L^\rho_T(L^p)}\right\|_{\ell^r}
$$
and denote by
$\tilde{L}^\rho_T(\dot{B}^s_{p,r})$ the subset of distributions
$u\in\mathcal{S}'([0,T]\times \mathbb{R}^N)$ with finite
$\|u\|_{\tilde{L}^\rho_T(\dot{B}^s_{p,r})}$ norm. When $T=+\infty$, the index $T$ is
omitted. We
further denote $\tilde{C}_T(\dot{B}^s_{p,r})=C([0,T];\dot{B}^s_{p,r})\cap
\tilde{L}^\infty_{T}(\dot{B}^s_{p,r}) $.
 \end{defn}
\begin{rem}\label{rem-CM-holder}
All the properties of continuity for the paraproduct, remainder, and product remain true for the Chemin-Lerner spaces. The exponent $\rho$ just has to behave according to H\"{o}lder's ineauality for the time variable.
\end{rem}

\begin{rem}\label{rem-CM-minkowski}
The spaces $\tl{L}^\rho_T(\dot{B}^s_{p,r})$ can be linked with the classical space $L^\rho_T(\dot{B}^s_{p,r})$ via the Minkowski inequality:
\beno
\|u\|_{\tl{L}^\rho_T(\dot{B}^s_{p,r})}\le\|u\|_{L^\rho_T(\dot{B}^s_{p,r})}\quad \mathrm{if}\quad r\ge\rho,\qquad \|u\|_{\tl{L}^\rho_T(\dot{B}^s_{p,r})}\ge\|u\|_{L^\rho_T(\dot{B}^s_{p,r})}\quad \mathrm{if}\quad r\le\rho.
\eeno
\end{rem}
\section{Linearized system}
\noindent We begin this section by giving the paralinearized version of system \eqref{COB2dimensionless}
\beq\label{piOB}
\begin{cases}
\pr_ta+\dv(\dot{T}_va)+\dv u=F,\\
\pr_t u+\dot{T}_v\cdot\nb u-\frac{1}{\Rey}\mathcal{A}u+\nb a-\frac{1}{\Rey}\dv\tau=G,\\
\pr_t\tau+\dot{T}_v\cdot\nb\tau+\frac{1}{\We}\tau-\frac{2\om}{\We}D( u)=L,
\end{cases}
\eeq
where $v, F, G$ and $L$ are some known functions. The purpose of this section is to establish the following property of system \eqref{piOB}.
\begin{prop}\label{prop-global}
Let $(a, u, \tau)$ be the solution to \eqref{piOB}. There exists a constant $C$ depending on $d, \Rey, \We$ and $\om$, such that for all $s\in\R$, we have
\beq\label{eq-prop}
\nn&&\|a(t)\|_{\dot{B}^{s-1,s}_{2,1}}+\|u(t)\|_{\dot{B}^{s-1}_{2,1}}+\|\tau(t)\|_{\dot{B}^{s}_{2,1}}\\
\nn&&+\|a\|_{L^1_t(\dot{B}^{s+1,s}_{2,1})}+\|u\|_{L^1_t(\dot{B}^{s+1}_{2,1})}+\|\tau\|_{L^1_t(\dot{B}^{s}_{2,1})}\\
\nn&\le&C\exp\left(C\|\nb v\|_{L^1_t(L^\infty)}\right)\left(\|a_0\|_{\dot{B}^{s-1,s}_{2,1}}+\|u_0\|_{\dot{B}^{s-1}_{2,1}}+\|\tau_0\|_{\dot{B}^{s}_{2,1}}\right.\\
&&\left.+\|F\|_{L^1(\dot{B}^{s-1,s}_{2,1})}+\|G\|_{L^1_t(\dot{B}^{s-1}_{2,1})}+\|L\|_{L^1_t(\dot{B}^{s}_{2,1})}\right).
\eeq
\end{prop}
\begin{proof}
Before proceeding any further, let us  localize the system \eqref{piOB}. Similar to \eqref{COB-c} and \eqref{COB-i}, we find that $(a_q, \Pe^\bot u_q, \Pe^\bot\dv\tau_q)$ and $(\Pe u_q, \Pe\dv\tau_q)$ satisfy respectively
\beq\label{lcOBq}
\begin{cases}
\pr_ta_q+\dv(v_qa_q)+\dv\Pe^\bot u_q=f_q,\\
\pr_t\Pe^\bot u_q+v_q\cdot\nb\Pe^\bot u_q-\frac{2(1-\om)}{\Rey}\Dl\Pe^\bot u_q+\nb a_q-\frac{1}{\Rey}\Pe^\bot\dv\tau_q=g_q^\sharp,\\
\pr_t\Pe^\bot\dv\tau_q+v_q\cdot\nb\Pe^\bot\dv\tau_q+\frac{1}{\We}\Pe^\bot\dv\tau_q-\frac{2\om}{\We}\Dl\Pe^\bot u_q=h_q^\sharp,
\end{cases}
\eeq
and
\beq\label{IiOBq}
\begin{cases}
\pr_t\Pe u_q+v_q\cdot\nb\Pe u_q-\frac{1-\om}{\Rey}\Dl\Pe u_q-\frac{1}{\Rey}\Pe\dv\tau_q=g_q^\flat,\\
\pr_t\Pe\dv\tau_q+v_q\cdot\nb\Pe\dv\tau_q+\frac{1}{\We}\Pe\dv\tau_q-\frac{\om}{\We}\Dl\Pe u_q=h_q^\flat,
\end{cases}
\eeq
where $v_q:=\dot{S}_{q-1}v$,
\begin{gather*}
f_q:=F_q+\dv(v_qa_q-\ddl\dot{T}_va),\\
g_q^\sharp:=\Pe^\bot G_q+\left(v_q\cdot\nb\ddl\Pe^\bot u-\Pe^\bot
\ddl\dot{T}_v\cdot\nb u\right),\\
h_q^\sharp:=\Pe^\bot H_q+\left(v_q\cdot\nb\ddl\Pe^\bot \dv\tau-\Pe^\bot
\ddl\dot{T}_v\cdot\nb \dv\tau\right),\\
g_q^\flat:=\Pe G_q+\left(v_q\cdot\nb\ddl\Pe u-\Pe
\ddl\dot{T}_v\cdot\nb u\right),\\
h_q^\flat:=\Pe H_q+\left(v_q\cdot\nb\ddl\Pe \dv\tau-\Pe
\ddl\dot{T}_v\cdot\nb \dv\tau\right),\\
\end{gather*}
with $H^k:=(\dv L)^k-\sum_{1\le i, j\le d}\dot{T}_{\pr_iv^j}\pr_j\tau^{i,k}$.

\noindent{\bf The Compressible Part}\par
\noindent Since the parabolic-hyperbolic system \eqref{lcOBq} behaves differently in high and low frequency, we have to deal with the high and low frequency part of \eqref{lcOBq}  in different ways. To simplify the presentation, in the following we give all the estimates which are needed both for high and low frequency cases. To begin with, taking the $L^2$ inner product of $\eqref{lcOBq}_1$ with $a_q$ and $\Dl a_q$ respectively, integrating by parts, we obtain
\be\label{la}
\frac12\frac{d}{dt}\|a_q\|_{L^2}^2+(\dv\Pe^\bot u_q|a_q)=(f_q|a_q)-\frac12\int\dv v_q|a_q|^2,
\ee
and
\be\label{ha}
\fr12\fr{d}{dt}\|\nb a_q\|_{L^2}^2+(\Dl\Pe^\bot u_q|\nb a_q)=(\nb f_q|\nb a_q)+\int a_q\pr_{ij}a_q\pr_i v_q^j-\fr12\int\dv v_q|\nb a_q|^2.
\ee
Next, taking the $L^2$ inner product of $\eqref{lcOBq}_2$ with $\Pe^\bot u_q$ and of  $\eqref{lcOBq}_3$ with $\Pe^\bot \dv\tau_q$ yields
\beq\label{Cu}
\nn\frac12\fr{d}{dt}\|\Pe^\bot u_q\|_{L^2}^2+\fr{2(1-\om)}{\Rey}\|\nb\Pe^\bot u_q\|_{L^2}^2+(\Pe^\bot u_q|\nb a_q)&-&\fr{1}{\Rey}(\Pe^\bot u_q|\Pe^\bot\dv\tau_q)\\
&=&(g_q^\sharp|\Pe^\bot u_q)+\fr12\int\dv v_q|\Pe^\bot u_q|^2,
\eeq
and
\beq\label{ctau}
\nn\fr12\fr{d}{dt}\|\Pe^\bot \dv\tau_q\|_{L^2}^2+\fr{1}{\We}\|\Pe^\bot\dv\tau_q\|^2_{L^2}&-&\fr{2\om}{\We}(\Dl\Pe^\bot u_q|\Pe^\bot\dv\tau_q)\\
&=&(h_q^\sharp|\Pe^\bot\dv\tau_q)+\fr12\int\dv v_q|\Pe^\bot\dv\tau_q|^2.
\eeq
In order to deal with the low frequency part, we also need to give the $L^2$ estimate of $\Lm^{-1}\Pe^\bot \dv\tau_q$. To do so, applying first the operator $\Lm^{-1}$ to third equation of \eqref{lcOBq}, and then taking the $L^2$ inner product of the resulting equation with $\Lm^{-1}\Pe^\bot\dv\tau_q$, we are let to
\beq\label{c-1tau}
\nn\fr12\fr{d}{dt}\|\Lm^{-1}\Pe^\bot\dv\tau_q\|_{L^2}^2&+&\fr{1}{\We}\|\Lm^{-1}\Pe^\bot\dv\tau_q\|^2_{L^2}+\fr{2\om}{\We}(\Pe^\bot u_q|\Pe^\bot\dv\tau_q)\\
&=&(\Lm^{-1}h_q^\sharp|\Lm^{-1}\Pe^\bot\dv\tau_q)-(\Lm^{-1}(v_q\cdot\nb\Pe^\bot\dv\tau_q)|\Lm^{-1}\Pe^{\bot}\dv\tau_q).
\eeq
Owing to the fact that the linear operator associated with \eqref{lcOBq} can not be diagonalized in a basis independent of $\xi$,  coercive estimates can not be achieved by means  of a linear combination of \eqref{la}--\eqref{c-1tau}. We must make full use the linear coupling terms in \eqref{lcOBq}. Accordingly,  the following three equalities of cross terms are given. More precisely, first of all, applying $\nb$ to the first equation of \eqref{lcOBq}, taking the $L^2$ inner product of the resulting equation with $\Pe^\bot u_q$. Secondly, taking the $L^2$ inner product of the second equation of \eqref{lcOBq} with $\nb a_q$. Summing up these two results, we find that
\beq\label{a-u}
\nn\fr{d}{dt}(\nb a_q|\Pe^\bot u_q)-\|\nb \Pe^\bot u_q\|_{L^2}^2+\|\nb a_q\|_{L^2}^2&-&\fr{2(1-\om)}{\Rey}(\Dl\Pe^\bot u_q|\nb a_q)-\fr{1}{\Rey}(\nb a_q|\Pe^\bot\dv\tau_q)\\
&=&(\nb f_q|\Pe^\bot u_q)+(g_q^\sharp|\nb a_q)+\int a_q\pr_jv_q^i\pr_i\Pe^\bot u_q^i.
\eeq
Similarly, we get the equalities involving the quantity $(\Pe^\bot u_q|\Pe^\bot\dv\tau_q)$ and $(\Pe^\bot u_q|\Dl \Pe^\bot \dv\tau_q)$,
\beq\label{cu-tau}
\nn&&\fr{d}{dt}(\Pe^\bot u_q|\Pe^\bot\dv\tau_q)+\fr{2\om}{\We}\|\nb\Pe^\bot u_q\|_{L^2}^2-\fr{1}{\Rey}\|\Pe^\bot\dv\tau_q\|_{L^2}^2\\
\nn&&-\fr{2(1-\om)}{\Rey}(\Dl\Pe^\bot u_q|\Pe^\bot\dv\tau_q)+(\nb a_q|\Pe^\bot\dv\tau_q)+\fr{1}{\We}(\Pe^\bot u_q|\Pe^\bot\dv\tau_q)\\
&=&(g_q^\sharp|\Pe^\bot\dv\tau_q)+(h_q^\sharp|\Pe^\bot u_q)+\int\dv v_q\Pe^\bot u_q\cdot\Pe^\bot\dv\tau_q,
\eeq
and
\beq\label{cu-Dltau}
&&\nn\fr{d}{dt}(\Pe^\bot u_q|\Dl \Pe^\bot \dv\tau_q)+\fr{1}{\Rey}\|\nb\Pe^\bot\dv\tau_q\|_{L^2}^2-\fr{2\om}{\We}\|\Dl\Pe^\bot u_q\|_{L^2}^2\\
&&\nn-\fr{2(1-\om)}{\Rey}(\Dl\Pe^\bot u_q|\Dl\Pe^\bot\dv\tau_q)+(\nb a_q|\Dl\Pe^\bot\dv\tau_q)+\fr{1}{\We}(\Dl \Pe^\bot u_q|\Pe^\bot\dv\tau_q)\\
&=&(g_q^\sharp|\Dl\Pe^\bot\dv\tau_q)+(h_q^\sharp|\Dl\Pe^\bot u_q)-(v_q\cdot\nb\Pe^\bot u_q|\Dl\Pe^\bot\dv\tau_q)-(v_q\cdot\nb\Pe^\bot\dv\tau_q|\Dl\Pe^\bot u_q).
\eeq
Now let us define
\beqno
Y_q&:=&\left[M\|\Pe^\bot u_q\|_{L^2}^2+2\Rey\We\left\|\fr{1-\om}{\Rey}\nb a_q\right\|^2_{L^2}+\left\|\fr{(1-\om)\We}{\om\Rey}\Pe^\bot\dv\tau_q\right\|^2_{L^2}\right.\\
&&\left.+2\Rey\We\left(\fr{1-\om}{\Rey}\nb a_q|\Pe^\bot u_q\right)-2\left(\Pe^\bot u_q|\fr{(1-\om)\We}{\om\Rey}\Pe^\bot\dv\tau_q\right)\right]^{\fr12},\quad \mathrm{for}\ q>q_0,
\eeqno
\beqno
Y_q&:=&\left[\|a_q\|_{L^2}^2+\|\Pe^\bot u_q\|_{L^2}^2+M'\|\fr{\We}{\Rey}\Pe^\bot\dv\tau_q\|_{L^2}^2+\fr{2}{\Rey}(\nb a_q|\Pe^\bot u_q)\right.\\
&&\left.+2(\Pe^\bot u_q|\fr{\We}{\Rey}\Pe^\bot\dv\tau_q)+2(\Pe^\bot u_q|\fr{\We}{\Rey}\Dl\Pe^\bot\dv\tau_q)\right]^{\fr12}, \quad \mathrm{for}\ q\le q_1,
\eeqno
and
\beqno
Y_q&:=&\left[\|a_q\|_{L^2}^2+\|\Pe^\bot u_q\|_{L^2}^2+\fr{\We}{2\om\Rey}\|\Lm^{-1}\Pe^\bot\dv\tau_q\|_{L^2}^2\right.\\
&&\left.+\fr{2\alpha(1-\om)}{\Rey}(\nb a_q|\Pe^\bot u_q)-\fr{\beta(1-\om)\We}{\om\Rey}(\Pe^\bot u_q|\Pe^\bot\dv\tau_q)\right]^{\fr12}, \quad \mathrm{for}\ q_1<q\le q_0,
\eeqno
where $q_0, q_1, M, M', \alpha$ and $\beta$ will be determined later. Next, we will estimate $Y_q$ for all $q\in \Z$ step by step.

\noindent{\em Step 1: high frequencies.}\par
\noindent Multiplying \eqref{ha} and \eqref{a-u} by $2\left(\fr{1-\om}{\Rey}\right)^2$ and $\fr{1-\om}{\Rey}$ respectively, summing up the resulting equations yields
\beq\label{eq-1}
\nn&&\fr12\fr{d}{dt}\left(2\left\|\fr{1-\om}{\Rey}\nb a_q\right\|^2+2\left(\fr{1-\om}{\Rey}\nb a_q|\Pe^\bot u_q\right)\right)\\
\nn&&-\fr{1-\om}{\Rey}\|\nb \Pe^\bot u_q\|_{L^2}^2+\fr{1-\om}{\Rey}\|\nb a_q\|_{L^2}^2-\fr{1}{\Rey}\left(\fr{1-\om}{\Rey}\nb a_q|\Pe^\bot\dv\tau_q\right)\\
&=&2\left(\fr{1-\om}{\Rey}\right)^2 {R.H.S.\ of\ } \eqref{ha}+\fr{1-\om}{\Rey}{R.H.S.\ of\ } \eqref{a-u}.
\eeq
\noindent Multiplying \eqref{ctau} and \eqref{cu-tau} by $\left(\fr{(1-\om)\We}{\om\Rey}\right)^2$ and $-\fr{(1-\om)\We}{\om\Rey}$ respectively, summing up the resulting equations, we are led to
\beq\label{eq-2}
\nn&&\fr12\fr{d}{dt}\left(\left\|\fr{(1-\om)\We}{\om\Rey}\Pe^\bot\dv\tau_q\right\|_{L^2}^2-2\left(\Pe^\bot u_q|\fr{(1-\om)\We}{\om\Rey}\Pe^\bot\dv\tau_q\right)\right)\\
\nn&&+\fr{1}{(1-\om)\We}\left\|\fr{(1-\om)\We}{\om\Rey}\Pe^\bot\dv\tau_q\right\|^2_{L^2}-\fr{2(1-\om)}{\Rey}\|\nb\Pe^\bot u_q\|_{L^2}^2\\
\nn&&-\left(\nb a_q|\fr{(1-\om)\We}{\om\Rey}\Pe^\bot\dv\tau_q\right)-\fr{1}{\We}\left(\Pe^\bot u_q|\fr{(1-\om)\We}{\om\Rey}\Pe^\bot\dv\tau_q\right)\\
&=&\left(\fr{(1-\om)\We}{\om\Rey}\right)^2 {R.H.S.\ of\ } \eqref{ctau}-\fr{(1-\om)\We}{\om\Rey}{R.H.S.\ of\ } \eqref{cu-tau}.
\eeq
Multiplying \eqref{Cu} and \eqref{eq-1} by $M$ and $\Rey\We$ respectively, adding them to \eqref{eq-2}, we obtain
\beq\label{eq-3}
\nn&&\fr12\fr{d}{dt}\left[M\|\Pe^\bot u_q\|_{L^2}^2+2\Rey\We\left\|\fr{1-\om}{\Rey}\nb a_q\right\|^2_{L^2}+\left\|\fr{(1-\om)\We}{\om\Rey}\Pe^\bot\dv\tau_q\right\|^2_{L^2}\right.\\
\nn&&\left.+2\Rey\We\left(\fr{1-\om}{\Rey}\nb a_q|\Pe^\bot u_q\right)-2\left(\Pe^\bot u_q|\fr{(1-\om)\We}{\om\Rey}\Pe^\bot\dv\tau_q\right)\right]\\
\nn&&-\fr{1}{\We}\left(\Pe^\bot u_q|\fr{(1-\om)\We}{\om\Rey}\Pe^\bot\dv\tau_q\right)+(2M-2-\Rey\We)\fr{1-\om}{\Rey}\|\nb\Pe^\bot u_q\|_{L^2}^2\\
\nn&&+(1-\om)\We\|\nb a_q\|_{L^2}^2+\fr{1}{(1-\om)\We}\left\|\fr{(1-\om)\We}{\om\Rey}\Pe^\bot\dv\tau_q\right\|^2_{L^2}\\
\nn&&-(1+\om)\left(\nb a_q|\fr{(1-\om)\We}{\om\Rey}\Pe^\bot\dv\tau_q\right)+M(\Pe^\bot u_q|\nb a_q)-\fr{M}{\Rey}(\Pe^\bot u_q|\Pe^\bot\dv\tau_q)\\
&=&M\, {R.H.S.\ of\ } \eqref{Cu}+\Rey\We{R.H.S.\ of\ } \eqref{eq-1}+{R.H.S.\ of\ } \eqref{eq-2}.
\eeq
By using Cauchy-Schwarz inequality, we have
\beno
2\Rey\We\left|\left(\fr{1-\om}{\Rey}\nb a_q|\Pe^\bot u_q\right)\right|\le \Rey\We\left(\left\|\fr{1-\om}{\Rey}\nb a_q\right\|^2_{L^2}+\|\Pe^\bot u_q\|_{L^2}^2\right),
\eeno
and
\beno
2\left|\left(\Pe^\bot u_q|\fr{(1-\om)\We}{\om\Rey}\Pe^\bot\dv\tau_q\right)\right|\le\fr32\|\Pe^\bot u_q\|_{L^2}^2+\fr23\left\|\fr{(1-\om)\We}{\om\Rey}\Pe^\bot\tau_q\right\|_{L^2}^2.
\eeno
It follows that
\beq\label{eq-4}
\nn&&M\|\Pe^\bot u_q\|_{L^2}^2+2\Rey\We\left\|\fr{1-\om}{\Rey}\nb a_q\right\|^2_{L^2}+\left\|\fr{(1-\om)\We}{\om\Rey}\Pe^\bot\dv\tau_q\right\|^2_{L^2}\\
\nn&&+2\Rey\We\left(\fr{1-\om}{\Rey}\nb a_q|\Pe^\bot u_q\right)-2\left(\Pe^\bot u_q|\fr{(1-\om)\We}{\om\Rey}\Pe^\bot\dv\tau_q\right)\\
&\ge&(M-\Rey\We-\fr32)\|\Pe^\bot u_q\|_{L^2}^2+\Rey\We\left\|\fr{1-\om}{\Rey}\nb a_q\right\|^2_{L^2}+\fr13\left\|\fr{(1-\om)\We}{\om\Rey}\Pe^\bot\dv\tau_q\right\|^2_{L^2}.
\eeq
Moreover,
\beno
(1+\om)\left(\nb a_q|\fr{(1-\om)\We}{\om\Rey}\Pe^\bot\dv\tau_q\right)\le \fr{1+\om}{2}\left((1-\om)\We\|\nb a_q\|_{L^2}^2+\fr{1}{(1-\om)\We}\left\|\fr{(1-\om)\We}{\om\Rey}\Pe^\bot\dv\tau_q\right\|^2_{L^2}\right),
\eeno
and hence
\beq\label{eq-5}
\nn&&(1-\om)\We\|\nb a_q\|_{L^2}^2+\fr{1}{(1-\om)\We}\left\|\fr{(1-\om)\We}{\om\Rey}\Pe^\bot\dv\tau_q\right\|^2_{L^2}\\
\nn&&-(1+\om)\left(\nb a_q|\fr{(1-\om)\We}{\om\Rey}\Pe^\bot\dv\tau_q\right)\\
\nn&\ge&\frac{(1-\om)}{2}\left((1-\om)\We\|\nb a_q\|_{L^2}^2+\fr{1}{(1-\om)\We}\left\|\fr{(1-\om)\We}{\om\Rey}\Pe^\bot\dv\tau_q\right\|^2_{L^2}\right)\\
&=&\fr{(1-\om)^2\We}{2}\|\nb a_q\|_{L^2}^2+\fr{1}{2\We}\left\|\fr{(1-\om)\We}{\om\Rey}\Pe^\bot\dv\tau_q\right\|^2_{L^2}.
\eeq
The rest terms on the left hand side of \eqref{eq-3} can be bounded as follows.
\beq\label{eq-6}
M|(\Pe^\bot u_q|\nb a_q)|\nn&\le&\fr14(1-\om)^2\We\|\nb a_q\|_{L^2}^2+\fr{M^2}{(1-\om)^2\We}\|\Pe^\bot u_q\|_{L^2}^2\\
&\le&\fr14(1-\om)^2\We\|\nb a_q\|_{L^2}^2+\left(\fr432^{-q_0}\right)^2\fr{M^2}{(1-\om)^2\We}\|\nb\Pe^\bot u_q\|_{L^2}^2,
\eeq
and
\beq\label{eq-7}
\nn&&\left|\fr{M}{\Rey}(\Pe^\bot u_q|\Pe^\bot\dv\tau_q)
+\fr{1}{\We}\left(\Pe^\bot u_q|\fr{(1-\om)\We}{\om\Rey}\Pe^\bot\dv\tau_q\right)\right|\\
\nn&\le&\fr{M}{(1-\om)\We}\left\|\fr{(1-\om)\We}{\om\Rey}\Pe^\bot\dv\tau_q\right\|_{L^2}\|\Pe^\bot u_q\|_{L^2}\\
\nn&\le&\fr14\fr{1}{\We}\left\|\fr{(1-\om)\We}{\om\Rey}\Pe^\bot\dv\tau_q\right\|_{L^2}^2+\fr{M^2}{(1-\om)^2\We}\|\Pe^\bot u_q\|_{L^2}^2\\
&\le&\fr14\fr{1}{\We}\left\|\fr{(1-\om)\We}{\om\Rey}\Pe^\bot\dv\tau_q\right\|_{L^2}^2+\left(\fr432^{-q_0}\right)^2\fr{M^2}{(1-\om)^2\We}\|\nb\Pe^\bot u_q\|_{L^2}^2.
\eeq
From \eqref{eq-5}--\eqref{eq-7}, we arrive at
\beq\label{eq-8}
\nn&&(2M-2-\Rey\We)\fr{1-\om}{\Rey}\|\nb\Pe^\bot u_q\|_{L^2}^2+(1-\om)\We\|\nb a_q\|_{L^2}^2\\
\nn&&+\fr{1}{(1-\om)\We}\left\|\fr{(1-\om)\We}{\om\Rey}\Pe^\bot\dv\tau_q\right\|^2_{L^2}
-(1+\om)\left(\nb a_q|\fr{(1-\om)\We}{\om\Rey}\Pe^\bot\dv\tau_q\right)\\
\nn&&+M(\Pe^\bot u_q|\nb a_q)-\fr{M}{\Rey}(\Pe^\bot u_q|\Pe^\bot\dv\tau_q)
-\fr{1}{\We}\left(\Pe^\bot u_q|\fr{(1-\om)\We}{\om\Rey}\Pe^\bot\dv\tau_q\right)\\
\nn&\ge&\left((2M-2-\Rey\We)\fr{1-\om}{\Rey}-\left(\fr432^{-q_0}\right)^2\fr{2M^2}{(1-\om)^2\We}\right)\|\nb\Pe^\bot u_q\|_{L^2}^2\\
&&+\fr14(1-\om)^2\We\|\nb a_q\|_{L^2}^2+\fr{1}{4\We}\left\|\fr{(1-\om)\We}{\om\Rey}\Pe^\bot\dv\tau_q\right\|^2_{L^2}.
\eeq
Taking $M=\Rey\We+2$, and
\be \label{eq-8+}
2^{q_0}\ge \left(\fr43\right)^{\fr32}\sqrt{\fr{2\Rey(\Rey\We+2)}{(1-\om)^3\We}},
\ee
then it is easy to verify that
\be
M-\Rey\We-\fr32\ge\fr12, \quad\mathrm{and}\quad (2M-2-\Rey\We)\fr{1-\om}{\Rey}-\left(\fr432^{-q_0}\right)^2\fr{2M^2}{(1-\om)^2\We}\ge\fr{\Rey\We+2}{4}\fr{1-\om}{\Rey}.
\ee
Next, we estimate the right hand side of \eqref{eq-3}. Indeed, using Cauchy-Schwarz and Bernstein's inequality, we have
\beq\label{RHS-CH}
\nn&& R.H.S. of \eqref{eq-3}\\
\nn&\le&C\left(\|\nb f_q\|_{L^2}+\|g_q^\sharp\|_{L^2}+\|h_q^\sharp\|_{L^2}\right)\left(\|\Pe^\bot u_q\|_{L^2}+\left\|\fr{1-\om}{\Rey}\nb a_q\right\|_{L^2}+\left\|\fr{(1-\om)\We}{\om\Rey}\Pe^\bot\dv\tau_q\right\|_{L^2}\right)\\
&&+C\|\nb v\|_{L^\infty}\left(\|\Pe^\bot u_q\|_{L^2}+\left\|\fr{1-\om}{\Rey}\nb a_q\right\|_{L^2}+\left\|\fr{(1-\om)\We}{\om\Rey}\Pe^\bot\dv\tau_q\right\|_{L^2}\right)^2.
\eeq
In view of commutator estimates, cf. \cite{Bahouri-Chemin-Danchin11}, we infer that
\beq\label{com-CH}
\nn&&\|\nb f_q\|_{L^2}+\|g_q^\sharp\|_{L^2}+\|h_q^\sharp\|_{L^2}\\
&\le& \|\nb F_q\|_{L^2}+\|\Pe^\bot G_q\|_{L^2}+\|\Pe^\bot H_q\|_{L^2}+ C\|\nb v\|_{L^\infty}\sum_{|q'-q|\le4}\left(\|\nb a_{q'}\|_{L^2}+\| u_{q'}\|_{L^2}+\|\dv\tau_{q'}\|_{L^2}\right).
\eeq
Now from \eqref{eq-4}, \eqref{eq-8}--\eqref{com-CH}, we find that there exist two constants $c_1$ and $C$ depending on $d, \Rey, \We$ and $\om$, such that
 if $q> q_0$,  then \eqref{eq-3} implies that
\beq\label{es-CH}
\nn\fr{d}{dt}Y_q+c_1Y_q
&\le& C\Bigg(\|\nb F_q\|_{L^2}+\|\Pe^\bot G_q\|_{L^2}+\|\Pe^\bot H_q\|_{L^2}\\
&&+\|\nb v\|_{L^\infty}\sum_{|q'-q|\le4}\left(\|\nb a_{q'}\|_{L^2}+\| u_{q'}\|_{L^2}+\|\dv\tau_{q'}\|_{L^2}\right)\Bigg).
\eeq
{\em Step 2: low frequencies.}\par
Part {\bf(i)}. $q\le q_1$. Multiplying \eqref{a-u} and \eqref{cu-tau} by $\fr{1}{\Rey}$ and $\fr{\We}{\Rey}$ respectively, and then adding them to
\eqref{la} and \eqref{Cu}, we get
\beq\label{eq-9}
\nn&&\fr12\fr{d}{dt}\left[\|a_q\|_{L^2}^2+\|\Pe^\bot u_q\|_{L^2}^2+\fr{2}{\Rey}(\nb a_q|\Pe^\bot u_q)+\fr{2\We}{\Rey}(\Pe^\bot u_q|\Pe^\bot\dv\tau_q)\right]\\
\nn&&+\fr{1}{\Rey}\|\nb\Pe^\bot u_q\|_{L^2}^2+\fr{1}{\Rey}\|\nb a_q\|_{L^2}^2-\fr{\We}{(\Rey)^2}\|\Pe^\bot\dv\tau_q\|_{L^2}^2-\fr{2(1-\om)}{(\Rey)^2}(\Dl\Pe^\bot u_q|\nb a_q)\\
\nn&&-\fr{2(1-\om)\We}{(\Rey)^2}(\Dl\Pe^\bot u_q|\Pe^\bot\dv\tau_q)-\fr{1}{(\Rey)^2}(\nb a_q|\Pe^\bot\dv\tau_q)+\fr{\We}{\Rey}(\nb a_q|\Pe^\bot\dv\tau_q)\\
&=&R.H.S. \ of\ \eqref{la} \ and\  \eqref{Cu}+\fr{1}{\Rey}R.H.S. \ of\ \eqref{a-u}+\fr{\We}{\Rey}R.H.S. \ of\ \eqref{cu-tau}.
\eeq
Multiplying \eqref{ctau} and \eqref{cu-Dltau} by $M'\left(\fr{\We}{\Rey}\right)^2$ and $\fr{\We}{\Rey}$ respectively, adding them to the above inequality yields
\beq\label{eq-10}
\nn&&\fr12\fr{d}{dt}\left[\|a_q\|_{L^2}^2+\|\Pe^\bot u_q\|_{L^2}^2+M'\|\fr{\We}{\Rey}\Pe^\bot\dv\tau_q\|_{L^2}^2+\fr{2}{\Rey}(\nb a_q|\Pe^\bot u_q)\right.\\
\nn&&\left.+2(\Pe^\bot u_q|\fr{\We}{\Rey}\Pe^\bot\dv\tau_q)+2(\Pe^\bot u_q|\fr{\We}{\Rey}\Dl\Pe^\bot\dv\tau_q)\right]+\fr{1}{\Rey}\|\nb\Pe^\bot u_q\|_{L^2}^2\\
\nn&&-\fr{2\om}{\Rey}\|\Dl\Pe^\bot u_q\|_{L^2}^2+\fr{1}{\Rey}\|\nb a_q\|_{L^2}^2+\fr{M'-1}{\We}\|\fr{\We}{\Rey}\Pe^\bot\dv\tau_q\|_{L^2}^2\\
\nn&&+\fr{1}{\We}\|\fr{\We}{\Rey}\nb\Pe^\bot\dv\tau_q\|^2_{L^2}-\fr{2(1-\om)}{(\Rey)^2}(\Dl\Pe^\bot u_q|\nb a_q)+(1-\fr{1}{\Rey\We})(\nb a_q|\fr{\We}{\Rey}\Pe^\bot\dv\tau_q)\\
\nn&&+\left(\fr{1}{\We}-\fr{2(1-\om)+2M'\om}{\Rey}\right)(\Dl\Pe^\bot u_q|\fr{\We}{\Rey}\Pe^\bot\dv\tau_q)\\
\nn&&-\fr{2(1-\om)}{\Rey}(\Dl\Pe^\bot u_q|\fr{\We}{\Rey}\Dl\Pe^\bot\dv\tau_q)+(\nb a_q|\fr{\We}{\Rey}\Dl\Pe^\bot\dv\tau_q)\\
&=&R.H.S. \ of\ \eqref{eq-9}+M'\left(\fr{\We}{\Rey}\right)^2R.H.S. \ of\ \eqref{ctau}+\fr{\We}{\Rey}R.H.S. \ of\ \eqref{cu-Dltau}.
\eeq
Now we estimate the cross terms on the left hand side of \eqref{eq-10}.
\beq\label{eq-11}
\nn\fr{2}{\Rey}|(\nb a_q|\Pe^\bot u_q)|&\le&\fr{8}{3}2^{q_1}\fr{1}{\Rey}\left(\|a_q\|_{L^2}^2+\|\Pe^\bot u_q\|_{L^2}^2\right)\\
&\le&\fr14\left(\|a_q\|_{L^2}^2+\|\Pe^\bot u_q\|_{L^2}^2\right),
\eeq
provided
\be\label{eq-12}
2^{q_1}\le \fr{3}{32}\Rey.
\ee
\be\label{eq-13}
2\left|(\Pe^\bot u_q|\fr{\We}{\Rey}\Pe^\bot\dv\tau_q)\right|\le \fr14\|\Pe^\bot u_q\|_{L^2}^2+4\left\|\fr{\We}{\Rey}\Pe^\bot\dv\tau_q\right\|_{L^2}^2.
\ee

\beq\label{eq-14}
\nn2\left|(\Pe^\bot u_q|\fr{\We}{\Rey}\Dl\Pe^\bot\dv\tau_q)\right|&\le&\left(\fr{8}{3}2^{q_1}\right)^2\left(\|\Pe^\bot u_q\|_{L^2}^2+\left\|\fr{\We}{\Rey}\Pe^\bot\dv\tau_q\right\|_{L^2}^2\right)\\
&\le&\fr14\left(\|\Pe^\bot u_q\|_{L^2}^2+\left\|\fr{\We}{\Rey}\Pe^\bot\dv\tau_q\right\|_{L^2}^2\right),
\eeq
provided
\be\label{eq-15}
2^{q_1}\le \fr{3}{16}.
\ee
From \eqref{eq-11}--\eqref{eq-15}, it is easy to see that
\beq\label{eq-16}
\nn&&\|a_q\|_{L^2}^2+\|\Pe^\bot u_q\|_{L^2}^2+M'\|\fr{\We}{\Rey}\Pe^\bot\dv\tau_q\|_{L^2}^2+\fr{2}{\Rey}(\nb a_q|\Pe^\bot u_q)\\
\nn&&+2(\Pe^\bot u_q|\fr{\We}{\Rey}\Pe^\bot\dv\tau_q)+2(\Pe^\bot u_q|\fr{\We}{\Rey}\Dl\Pe^\bot\dv\tau_q)\\
&\ge&\fr34\|a_q\|_{L^2}^2+\fr14\|\Pe^\bot u_q\|_{L^2}^2+(M'-\fr{17}{4})\left\|\fr{\We}{\Rey}\Pe^\bot\dv\tau_q\right\|_{L^2}^2.
\eeq
Using Cauchy-Schwarz inequality and Lemma \ref{Bernstein} over and over again, the rest cross terms on the left hand side of \eqref{eq-10} can be bounded in a similar way.  In fact,
\beq\label{eq-17}
\nn\fr{2(1-\om)}{(\Rey)^2}|(\Dl\Pe^\bot u_q|\nb a_q)|&\le&\fr{2(1-\om)}{(\Rey)^2}2^{q_1}\fr{8}{3}\|\nb\Pe^\bot u_q\|_{L^2}\|\nb a_q\|_{L^2}\\
\nn&\le&\fr{8}{3\Rey}2^{q_1}\left(\fr{1}{\Rey}\|\nb\Pe^\bot u_q\|_{L^2}^2+\fr{1}{\Rey}\|\nb a_q\|_{L^2}^2\right)\\
&\le&\fr14\left(\fr{1}{\Rey}\|\nb\Pe^\bot u_q\|_{L^2}^2+\fr{1}{\Rey}\|\nb a_q\|_{L^2}^2\right),
\eeq
provided
\be\label{eq-18}
2^{q_1}\le\fr{3}{32}\Rey.
\ee

\be\label{eq-19}
\left|(1-\fr{1}{\Rey\We})(\nb a_q|\fr{\We}{\Rey}\Pe^\bot\dv\tau_q)\right|\le\fr14\fr{1}{\Rey}\|\nb a_q\|_{L^2}^2+\fr{(\Rey\We-1)^2}{\Rey\We}\fr{1}{\We}\left\|\fr{\We}{\Rey}\Pe^\bot\dv\tau_q\right\|_{L^2}^2.
\ee

\beq\label{eq-20}
\nn&&\left|\left(\fr{1}{\We}-\fr{2(1-\om)+2M'\om}{\Rey}\right)(\Dl\Pe^\bot u_q|\fr{\We}{\Rey}\Pe^\bot\dv\tau_q)\right|\\
\nn&\le&(\fr{1}{\We}+\fr{2M'}{\Rey})\|\Dl\Pe^\bot u_q\|_{L^2}\left\|\fr{\We}{\Rey}\Pe^\bot\dv\tau_q\right\|_{L^2}\\
&\le&\fr14\fr{M'}{\We}\left\|\fr{\We}{\Rey}\Pe^\bot\dv\tau_q\right\|_{L^2}^2+\left(\fr832^{q_1}\right)^2
\left(\fr{4M'^2\We}{(M'-1)\Rey}+\fr{\Rey}{\We}\right)\fr{1}{\Rey}\|\nb\Pe^\bot u_q\|_{L^2}^2.
\eeq

\beq\label{eq-21}
\nn&&\fr{2(1-\om)}{\Rey}\left|(\Dl\Pe^\bot u_q|\fr{\We}{\Rey}\Dl\Pe^\bot\dv\tau_q)\right|\\
\nn&\le&\fr{2}{\Rey}\left(\fr83 2^{q_1}\right)^3\|\nb\Pe^\bot u_q\|_{L^2}\|\fr{\We}{\Rey}\Pe^\bot\dv\tau_q\|_{L^2}\\
&\le&\fr{1}{4}\fr{1}{\We}\left\|\fr{\We}{\Rey}\Pe^\bot\dv\tau_q\right\|_{L^2}^2+\fr{4\We}{\Rey}\left(2^{q_1}\fr83\right)^6\fr{1}{\Rey}\|\nb\Pe^\bot u_q\|_{L^2}^2.
\eeq

\beq\label{eq-22}
\nn&&\left|(\nb a_q|\fr{\We}{\Rey}\Dl\Pe^\bot\dv\tau_q)\right|\\
\nn&\le&\left(\fr83 2^{q_1}\right)^2\|\nb a_q\|_{L^2}\left\|\fr{\We}{\Rey}\Pe^\bot\dv\tau_q\right\|_{L^2}\\
&\le&\fr{1}{4}\fr{1}{\We}\left\|\fr{\We}{\Rey}\Pe^\bot\dv\tau_q\right\|_{L^2}^2+\Rey\We\left(\fr83 2^{q_1}\right)^4\fr{1}{\Rey}\|\nb a_q\|_{L^2}^2.
\eeq
 Moreover,
\be\label{eq-23}
\fr{2\om}{\Rey}\|\Dl\Pe^\bot u_q\|_{L^2}^2\le2\left(\fr83 2^{q_1}\right)^2\fr{1}{\Rey}\|\nb\Pe^\bot u_q\|^2_{L^2}.
\ee
It follows from \eqref{eq-17}--\eqref{eq-23} that
\beq\label{eq-24}
\nn&&\fr{1}{\Rey}\|\nb\Pe^\bot u_q\|_{L^2}^2-\fr{2\om}{\Rey}\|\Dl\Pe^\bot u_q\|_{L^2}^2+\fr{1}{\Rey}\|\nb a_q\|_{L^2}^2+\fr{M'-1}{\We}\left\|\fr{\We}{\Rey}\Pe^\bot\dv\tau_q\right\|_{L^2}^2\\
\nn&&+\fr{1}{\We}\|\fr{\We}{\Rey}\nb\Pe^\bot\dv\tau_q\|^2_{L^2}-\fr{2(1-\om)}{(\Rey)^2}(\Dl\Pe^\bot u_q|\nb a_q)+(1-\fr{1}{\Rey\We})(\nb a_q|\fr{\We}{\Rey}\Pe^\bot\dv\tau_q)\\
\nn&&+\left(\fr{1}{\We}-\fr{2(1-\om)+2M'\om}{\Rey}\right)(\Dl\Pe^\bot u_q|\fr{\We}{\Rey}\Pe^\bot\dv\tau_q)\\
\nn&&-\fr{2(1-\om)}{\Rey}(\Dl\Pe^\bot u_q|\fr{\We}{\Rey}\Dl\Pe^\bot\dv\tau_q)+(\nb a_q|\fr{\We}{\Rey}\Dl\Pe^\bot\dv\tau_q)\\
\nn&\ge&\left(\fr34-\left(\fr832^{q_1}\right)^2
\left(\fr{4M'^2\We}{(M'-1)\Rey}+\fr{\Rey}{\We}\right)-\fr{4\We}{\Rey}\left(2^{q_1}\fr83\right)^6-2\left(\fr83 2^{q_1}\right)^2\right)\fr{1}{\Rey}\|\nb\Pe^\bot u_q\|^2_{L^2}\\
\nn&&+\left(\fr12-\Rey\We\left(\fr832^{q_1}\right)^4\right)\fr{1}{\Rey}\|\nb a_q\|_{L^2}^2+\fr{1}{\We}\left\|\fr{\We}{\Rey}\nb\Pe^\bot\dv\tau_q\right\|_{L^2}^2\\
&&+\left(\fr{3M'}{4}-\fr32-\fr{(\Rey\We-1)^2}{\Rey\We}\right)\fr{1}{\We}\left\|\fr{\We}{\Rey}\Pe^\bot\dv\tau_q\right\|_{L^2}^2.
\eeq
Taking
\be\label{eq-25}
M':=\fr43\left(\Rey\We+\fr{1}{\Rey\We}\right)+2,
\ee
and
\be\label{eq-26}
2^{q_1}\le\fr{3}{32}\min\left(1,\Rey, \fr{2}{\left(\Rey\We\right)^{\fr14}}, \left(\fr{\We}{\Rey}\right)^{\fr12}, 2\left(\fr{\Rey}{\We}\right)^{\fr16}, \fr{\sqrt{M'-1}}{2M'}\cdot\sqrt{\fr{\Rey}{\We}}\right),
\ee
then \eqref{eq-12}, \eqref{eq-15} and \eqref{eq-18} hold, and
\begin{gather*}
\Rey\We\left(\fr832^{q_1}\right)^4\le\fr{1}{16}, \qquad \left(\fr832^{q_1}\right)^2\fr{4M'^2\We}{(M'-1)\Rey}\le\fr{1}{16},\\
\left(\fr832^{q_1}\right)^2\fr{\Rey}{\We}\le\fr{1}{16}, \quad \fr{4\We}{\Rey}\left(\fr832^{q_1}\right)^6\le\fr{1}{16},\quad 2\left(\fr832^{q_1}\right)^2\le\fr18.
\end{gather*}
Accordingly,
\beq\label{eq-27}
M'-\fr{17}{4}\ge\fr{5}{12},
\eeq

\beq\label{eq-28}
\fr34-\left(\fr832^{q_1}\right)^2
\left(\fr{4M'^2\We}{(M'-1)\Rey}+\fr{\Rey}{\We}\right)-\fr{4\We}{\Rey}\left(2^{q_1}\fr83\right)^6-2\left(\fr83 2^{q_1}\right)^2\ge\fr{7}{16},
\eeq

\be\label{eq-28+}
\fr12-\Rey\We\left(\fr832^{q_1}\right)^4\ge\fr{7}{16}, \quad \mathrm{and} \quad \fr{3M'}{4}-\fr32-\fr{(\Rey\We-1)^2}{\Rey\We}\ge2.
\ee
Next, we estimate the right hand side of \eqref{eq-10}. To this end, notice first that integrating by parts yields
\beq\label{eq-29}
 \nn&&-(v_q\cdot\nb\Pe^\bot u_q|\Dl\Pe^\bot\dv\tau_q)-(v_q\cdot\nb\Pe^\bot\dv\tau_q|\Dl\Pe^\bot u_q)\\
 &=&(\nb v_q\nb\Pe^\bot u_q|\nb \Pe^\bot\dv\tau_q)+(\nb v_q\nb\dv\tau_q|\nb\Pe^\bot u_q)-(\dv v_q|\nb\Pe^\bot u_q:\nb\Pe^\bot\dv\tau_q).
 \eeq
Substituting this equality to \eqref{eq-10}, using the fact $q\le q_1$, it is not difficult to verify that
\beq\label{eq-30}
\nn&& R. H. S. of\, \eqref{eq-10}\\
\nn&\le&C\left(\|a_q\|_{L^2}+\|\Pe^\bot u_q\|_{L^2}+\left\|\fr{\We}{\Rey}\Pe^\bot\dv\tau_q\right\|_{L^2}\right)\\
&&\cdot\left[\left(\|f_q\|_{L^2}+\|g_q^\sharp\|_{L^2}+\|h_q^\sharp\|_{L^2}\|\right)+\|\nb v_q\|_{L^\infty}\left(\|a_q\|_{L^2}+\|\Pe^\bot u_q\|_{L^2}+\left\|\fr{\We}{\Rey}\Pe^\bot\dv\tau_q\right\|_{L^2}\right)\right].
\eeq
Similar to \eqref{com-CH}, we have
\beq\label{com-CLL}
\nn&&\|f_q\|_{L^2}+\|g_q^\sharp\|_{L^2}+\|h_q^\sharp\|_{L^2}\\
&\le& \|F_q\|_{L^2}+\|\Pe^\bot G_q\|_{L^2}+\|\Pe^\bot H_q\|_{L^2}+ C\|\nb v\|_{L^\infty}\sum_{|q'-q|\le4}\left(\| a_{q'}\|_{L^2}+\| u_{q'}\|_{L^2}+\|\dv\tau_{q'}\|_{L^2}\right).
\eeq
Now we infer from \eqref{eq-16}, \eqref{eq-24}--\eqref{eq-28+}, \eqref{eq-30} and \eqref{com-CLL} that there exist two constants $c_2$ and $C$
depending on $d, \Rey, \We$ and $\om$, such that if $q\le q_1$, then \eqref{eq-10} implies
\beq\label{es-CLL}
\fr{d}{dt}Y_q+c_22^{2q}Y_q
\nn&\le& C\Bigg(\|F_q\|_{L^2}+\|\Pe^\bot G_q\|_{L^2}+\|\Pe^\bot H_q\|_{L^2}\\
&&+\|\nb v\|_{L^\infty}\sum_{|q'-q|\le4}\left(\| a_{q'}\|_{L^2}+\| u_{q'}\|_{L^2}+\|\dv\tau_{q'}\|_{L^2}\right)\Bigg).
\eeq

Part ({\bf ii}). $q_1<q\le q_0$. Multiplying \eqref{c-1tau} by $\fr{\We}{2\om\Rey}$, adding the resulting equation to \eqref{la} and
\eqref{Cu}, we arrive at
\beq\label{eq-31}
\nn&&\fr12\fr{d}{dt}\left(\|a_q\|_{L^2}^2+\|\Pe^\bot u_q\|_{L^2}^2+\fr{\We}{2\om\Rey}\|\Lm^{-1}\Pe^\bot\dv\tau_q\|_{L^2}^2\right)\\
\nn&&+\fr{2(1-\om)}{\Rey}\|\nb\Pe^\bot u_q\|_{L^2}^2+\fr{1}{2\om\Rey}\|\Lm^{-1}\Pe^\bot\dv\tau_q\|_{L^2}^2\\
&\le&R.H.S.\ of \ \eqref{la}\ and\ \eqref{Cu}+\fr{\We}{2\om\Rey}R.H.S.\ of \ \eqref{c-1tau}.
\eeq
Multiplying \eqref{a-u} and \eqref{cu-tau} by $\al\fr{1-\om}{\Rey}$ and$-\beta\fr{1-\om}{2\om}\fr{\We}{\Rey}$ respectively, adding them to \eqref{eq-31}, we are led to
\beq\label{eq-32}
\nn&&\fr12\fr{d}{dt}\left(\|a_q\|_{L^2}^2+\|\Pe^\bot u_q\|_{L^2}^2+\fr{\We}{2\om\Rey}\|\Lm^{-1}\Pe^\bot\dv\tau_q\|_{L^2}^2\right.\\
\nn&&\left.+\fr{2\alpha(1-\om)}{\Rey}(\nb a_q|\Pe^\bot u_q)-\fr{\beta(1-\om)\We}{\om\Rey}(\Pe^\bot u_q|\Pe^\bot\dv\tau_q)\right)+\fr{\al(1-\om)}{\Rey}\|\nb a_q\|_{L^2}^2\\
\nn&&+(2-\al-\beta)\fr{1-\om}{\Rey}\|\nb\Pe^\bot u_q\|_{L^2}^2+\fr{\beta(1-\om)\We}{2\om(\Rey)^2}\|\Pe^\bot\dv\tau_q\|_{L^2}^2\\
\nn&&+\fr{1}{2\om\Rey}\|\Lm^{-1}\Pe^\bot\dv\tau_q\|_{L^2}^2-\fr{2\al(1-\om)^2}{(\Rey)^2}(\Dl\Pe^\bot u_q|\nb a_q)\\
\nn&&-\fr{\al(1-\om)}{(\Rey)^2}(\nb a_q|\Pe^\bot\dv\tau_q)+\fr{\beta(1-\om)^2\We}{\om(\Rey)^2}(\Dl\Pe^\bot u_q|\Pe^\bot\dv\tau_q)\\
\nn&&-\fr{\beta(1-\om)\We}{2\om\Rey}(\nb a_q|\Pe^\bot\dv\tau_q)-\fr{\beta(1-\om)}{2\om\Rey}(\Pe^\bot u_q|\Pe^\bot\dv\tau_q)\\
&\le& \fr{\al(1-\om)}{\Rey} R.H. S. \ of\ \eqref{a-u}-\fr{\beta(1-\om)\We}{2\om\Rey}R. H. S. \ of\ \eqref{cu-tau}+R. H. S. \ of\ \eqref{eq-31}.
\eeq
Now we estimate the cross terms in the left hand side of \eqref{eq-32} one by one. Indeed, all of them can be bounded by using Cauchy-Schwarz inequality and Lemma \ref{Bernstein}. More precisely,
\beq\label{eq-33}
\nn\fr{2\alpha(1-\om)}{\Rey}|(\nb a_q|\Pe^\bot u_q)|&\le&\fr{8}{3}2^{q_0}\fr{\alpha}{\Rey}\left(\|a_q\|_{L^2}^2+\|\Pe^\bot u_q\|_{L^2}^2\right)\\
&\le&\fr14\left(\|a_q\|_{L^2}^2+\|\Pe^\bot u_q\|_{L^2}^2\right),
\eeq
provided
\be\label{eq-34}
\al\le2^{-q_0}\fr{3}{32}\Rey.
\ee

\beq\label{eq-35}
\nn\fr{\beta(1-\om)\We}{\om\Rey}|(\Pe^\bot u_q|\Pe^\bot\dv\tau_q)|&\le&\fr{\beta\We}{\om\Rey}\fr832^{q_0}\|\Pe^\bot u_q\|_{L^2}\|\Lm^{-1}\Pe^\bot\dv\tau_q\|_{L^2}\\
\nn&\le&\fr14\|\Pe^\bot u_q\|_{L^2}^2+\left(\fr{\beta\We}{\om\Rey}\fr832^{q_0}\right)^2\|\Lm^{-1}\Pe^\bot\dv\tau_q\|_{L^2}^2\\
&\le&\fr14\|\Pe^\bot u_q\|_{L^2}^2+\fr14\fr{\We}{\om\Rey}\|\Lm^{-1}\Pe^\bot\dv\tau_q\|_{L^2}^2,
\eeq
provided
\be\label{eq-36}
\beta\le\fr{3}{16}2^{-q_0}\sqrt{\fr{\om\Rey}{\We}}.
\ee
Consequently, if \eqref{eq-34} and \eqref{eq-36} hold, we have
\beq\label{eq-37}
\nn&&\|a_q\|_{L^2}^2+\|\Pe^\bot u_q\|_{L^2}^2+\fr{\We}{2\om\Rey}\|\Lm^{-1}\Pe^\bot\dv\tau_q\|_{L^2}^2\\
\nn&&+\fr{2\alpha(1-\om)}{\Rey}(\nb a_q|\Pe^\bot u_q)-\fr{\beta(1-\om)\We}{\om\Rey}(\Pe^\bot u_q|\Pe^\bot\dv\tau_q)\\
&\ge&\fr34\|a_q\|_{L^2}^2+\fr12\|\Pe^\bot u_q\|_{L^2}^2+\fr14\fr{\We}{\om\Rey}\|\Lm^{-1}\Pe^\bot\dv\tau_q\|_{L^2}^2.
\eeq
Moreover,
\beq\label{eq-38}
\nn\fr{2\al(1-\om)^2}{(\Rey)^2}|(\Dl\Pe^\bot u_q|\nb a_q)|&\le&\left(\sqrt{\al}\fr832^{q_0}\fr{1}{\Rey}\right)\left[\fr{1-\om}{\Rey}\left(\|\nb\Pe^\bot u_q\|_{L^2}^2+\al\|\nb a_q\|_{L^2}^2\right)\right]\\
&\le&\fr14\left[\fr{1-\om}{\Rey}\left(\|\nb\Pe^\bot u_q\|_{L^2}^2+\al\|\nb a_q\|_{L^2}^2\right)\right],
\eeq
provided
\be\label{eq-39}
\al\le\left(2^{-q_0}\fr{3}{32}\Rey\right)^2.
\ee

\beq\label{eq-40}
\fr{\al(1-\om)}{(\Rey)^2}|(\nb a_q|\Pe^\bot\dv\tau_q)|\nn&\le&\left(\sqrt{\al}\fr432^{q_0}\fr{1}{\Rey}\right)\left(\fr{\al(1-\om)}{\Rey}\|\nb a_q\|_{L^2}^2+\fr{1}{\om\Rey}\|\Lm^{-1}\Pe^\bot\dv\tau_q\|_{L^2}^2\right)\\
&\le&\fr18\left(\fr{\al(1-\om)}{\Rey}\|\nb a_q\|_{L^2}^2+\fr{1}{\om\Rey}\|\Lm^{-1}\Pe^\bot\dv\tau_q\|_{L^2}^2\right),
\eeq
provided
\beq\label{eq-41}
\al\le\left(2^{-q_0}\fr{3}{32}\Rey\right)^2.
\eeq

\beq\label{eq-42}
\nn&&\fr{\beta(1-\om)^2\We}{\om(\Rey)^2}|(\Dl\Pe^\bot u_q|\Pe^\bot\dv\tau_q)|\\
\nn&\le&\fr14\fr{\beta(1-\om)\We}{\om(\Rey)^2}\|\Pe^\bot\dv\tau_q\|_{L^2}^2+\fr{\beta(1-\om)\We}{\om(\Rey)^2}\left(\fr832^{q_0}\right)^2
\|\nb\Pe\bot u_q\|_{L^2}^2\\
&\le&\fr14\fr{\beta(1-\om)\We}{\om(\Rey)^2}\|\Pe^\bot\dv\tau_q\|_{L^2}^2+\fr14\fr{1-\om}{\Rey}\|\nb \Pe^\bot u_q\|_{L^2}^2,
\eeq
provided
\be\label{eq-43}
\beta\le\left(\fr832^{q_0}\right)^{-2}\fr{\om\Rey}{4\We}.
\ee

\beq\label{eq-44}
\nn&&\fr{\beta(1-\om)\We}{2\om\Rey}|(\nb a_q|\Pe^\bot\dv\tau_q)|\\
&\le&\fr14\fr{\beta(1-\om)\We}{2\om(\Rey)^2}\|\Pe^\bot\dv\tau_q\|_{L^2}^2+\fr14\fr{\al(1-\om)}{\Rey}\|\nb a_q\|_{L^2}^2,
\eeq
provided
\be\label{eq-45}
\beta\le\fr{\om\al}{2\Rey\We}.
\ee

\beq\label{eq-46}
\nn&&\fr{\beta(1-\om)}{2\om\Rey}|(\Pe^\bot u_q|\Pe^\bot\dv\tau_q)|\\
\nn&=&\fr{\beta(1-\om)}{2\om\Rey}|(\Lm\Pe^\bot u_q|\Lm^{-1}\Pe^\bot\dv\tau_q)|\\
\nn&\le&\fr14\fr{1}{2\om\Rey}\|\Lm^{-1}\Pe^\bot\dv\tau_q\|_{L^2}^2+\fr{\beta^2}{2\om}\fr{1-\om}{\Rey}\|\nb\Pe^\bot u_q\|_{L^2}^2\\
&\le&\fr14\fr{1}{2\om\Rey}\|\Lm^{-1}\Pe^\bot\dv\tau_q\|_{L^2}^2+\fr14\fr{1-\om}{\Rey}\|\nb\Pe^\bot u_q\|_{L^2}^2,
\eeq
provided
\be\label{eq-47}
\beta\le\sqrt{\fr{\om}{2}}.
\ee
Collecting \eqref{eq-34}, \eqref{eq-36}, \eqref{eq-39}, \eqref{eq-41}, \eqref{eq-43}, \eqref{eq-45} and\eqref{eq-47}, we choose
\be\label{eq-48}
\al\le\min\left(\fr14, 2^{-q_0}\fr{3}{32}\Rey, \left(2^{-q_0}\fr{3}{32}\Rey\right)^2 \right)
\ee
and
\be\label{eq-49}
\beta\le\min\left(\fr14, 2^{-q_0}\fr{3}{16}\sqrt{\fr{\om\Rey}{\We}}, \left(\fr832^{q_0}\right)^{-2}\fr{\om\Rey}{4\We}, \fr{\om\al}{2\Rey\We}, \sqrt{\fr{\om}{2}}\right).
\ee
Then \eqref{eq-37}  and the following inequality hold,
\beq\label{eq-50}
\nn&&\fr{\al(1-\om)}{\Rey}\|\nb a_q\|_{L^2}^2+(2-\al-\beta)\fr{1-\om}{\Rey}\|\nb\Pe^\bot u_q\|_{L^2}^2+\fr{\beta(1-\om)\We}{2\om(\Rey)^2}\|\Pe^\bot\dv\tau_q\|_{L^2}^2\\
\nn&&+\fr{1}{2\om\Rey}\|\Lm^{-1}\Pe^\bot\dv\tau_q\|_{L^2}^2-\fr{2\al(1-\om)^2}{(\Rey)^2}(\Dl\Pe^\bot u_q|\nb a_q)\\
\nn&&-\fr{\al(1-\om)}{(\Rey)^2}(\nb a_q|\Pe^\bot\dv\tau_q)+\fr{\beta(1-\om)^2\We}{\om(\Rey)^2}(\Dl\Pe^\bot u_q|\Pe^\bot\dv\tau_q)\\
\nn&&-\fr{\beta(1-\om)\We}{2\om\Rey}(\nb a_q|\Pe^\bot\dv\tau_q)-\fr{\beta(1-\om)}{2\om\Rey}(\Pe^\bot u_q|\Pe^\bot\dv\tau_q)\\
\nn&\ge&\fr{3\al}{8}\fr{1-\om}{\Rey}\|\nb a_q\|_{L^2}^2+\fr34\fr{1-\om}{\Rey}\|\nb\Pe^\bot u_q\|_{L^2}^2\\
&&+\fr{\beta(1-\om)\We}{8\om(\Rey)^2}\|\Pe^\bot\dv\tau_q\|_{L^2}^2+\fr{1}{4\om\Rey}\|\Lm^{-1}\Pe^\bot\dv\tau_q\|_{L^2}^2.
\eeq
Now we are left to bound the right hand side of \eqref{eq-32}. To do so, noting first that
\beqno
&&(\Lm^{-1}(v_q\cdot\nb\Pe^\bot\dv\tau_q)|\Lm^{-1}\Pe^{\bot}\dv\tau_q)\\
&=&([\Lm^{-1},v_q\cdot\nb]\Pe^\bot\dv\tau_q)|\Lm^{-1}\Pe^{\bot}\dv\tau_q)
-\fr12\int\dv v_q|\Lm^{-1}\Pe^\bot\dv\tau_q|^2,
\eeqno
then Lemma \ref{lem-commu} implies
\be\label{eq-51}
|(\Lm^{-1}(v_q\cdot\nb\Pe^\bot\dv\tau_q)|\Lm^{-1}\Pe^{\bot}\dv\tau_q)|\le C\|\nb v_q\|_{L^\infty}\|\Lm^{-1}\Pe^\bot\dv\tau_q\|_{L^2}^2.
\ee
As a result, by virtue of Cauchy-Schwarz inequality and the fact $q\le q_0$, we arrive at
\beq\label{eq-52}
\nn &&R.H.S. of \eqref{eq-32}\\
\nn&\le&C\left(\|f_q\|_{L^2}+\|g_q^{\sharp}\|_{L^2}+\|\Lm^{-1}h_q^{\sharp}\|_{L^2}\right)\left(\|a_q\|_{L^2}+\|\Pe^\bot u_q\|_{L^2}+\sqrt{\fr{\We}{\om\Rey}}\|\Lm^{-1}\Pe^\bot\dv\tau_q\|_{L^2}\right)\\
&&+C\|\nb v_q\|_{L^\infty}\left(\|a_q\|_{L^2}+\|\Pe^\bot u_q\|_{L^2}+\sqrt{\fr{\We}{\om\Rey}}\|\Lm^{-1}\Pe^\bot\dv\tau_q\|_{L^2}\right)^2.
\eeq
Using the commutator estimates in \cite{Bahouri-Chemin-Danchin11} once more, one easily deduces that
\beq\label{eq-53}
\nn&&\|f_q\|_{L^2}+\|g_q^{\sharp}\|_{L^2}+\|\Lm^{-1}h_q^{\sharp}\|_{L^2}\\
\nn&\le&C\Bigg(\|F_q\|_{L^2}+\|\Pe^\bot G_q\|_{L^2}+\|\Lm^{-1}\Pe^\bot H_q\|_{L^2}\\
&&+\|\nb v\|_{L^\infty}\sum_{|q-q'|\le4}\left(\| a_{q'}\|_{L^2}+\| u_{q'}\|_{L^2}+\|\Lm^{-1}\dv\tau_{q'}\|_{L^2}\right)\Bigg).
\eeq
From \eqref{eq-37}, \eqref{eq-50} and \eqref{eq-52}--\eqref{eq-53}, we conclude  that there exist two constants $c_3$ and $C$ depending
on $d, \Rey, \We$ and $\om$, such that if $q_1<q\le q_0$ and $\al, \beta$ satisfy \eqref{eq-48} and \eqref{eq-49}, then \eqref{eq-32} implies
\beq\label{es-CLH}
\fr{d}{dt}Y_q+c_32^{2q}Y_q
\nn&\le& C\Bigg(\|F_q\|_{L^2}+\|\Pe^\bot G_q\|_{L^2}+\|\Lm^{-1}\Pe^\bot H_q\|_{L^2}\\
&&+\|\nb v\|_{L^\infty}\sum_{|q'-q|\le4}\left(\| a_{q'}\|_{L^2}+\| u_{q'}\|_{L^2}+\|\Lm^{-1}\dv\tau_{q'}\|_{L^2}\right)\Bigg).
\eeq

\bigbreak
\noindent{\bf The Incompressible Part} \par
\noindent We begin this part by giving the following five equalities in which the $L^2$ estimates and cross terms of the corresponding
incompressible part $(\Pe u_q,\Pe\dv\tau_q)$ of $(u_q,\dv\tau_q)$ are involved.  Since they are obtained in the same way with those in the
compressible part, we give a list of the results directly.
\be\label{eq-54}
\frac12\fr{d}{dt}\|\Pe u_q\|_{L^2}^2+\fr{1-\om}{\Rey}\|\nb\Pe u_q\|_{L^2}^2-\fr{1}{\Rey}(\Pe u_q|\Pe\dv\tau_q)
=(g_q^\flat|\Pe u_q)+\fr12\int\dv v_q|\Pe u_q|^2.
\ee

\beq\label{eq-55}
\nn\fr12\fr{d}{dt}\|\Pe \dv\tau_q\|_{L^2}^2&+&\fr{1}{\We}\|\Pe\dv\tau_q\|^2_{L^2}-\fr{\om}{\We}(\Dl\Pe u_q|\Pe\dv\tau_q)\\
&=&(h_q^\flat|\Pe\dv\tau_q)+\fr12\int\dv v_q|\Pe\dv\tau_q|^2.
\eeq

\beq\label{eq-56}
\nn\fr12\fr{d}{dt}\|\Lm^{-1}\Pe\dv\tau_q\|_{L^2}^2&+&\fr{1}{\We}\|\Lm^{-1}\Pe\dv\tau_q\|^2_{L^2}+\fr{\om}{\We}(\Pe u_q|\Pe\dv\tau_q)\\
&=&(\Lm^{-1}h_q^\flat|\Lm^{-1}\Pe\dv\tau_q)-(\Lm^{-1}(v_q\cdot\nb\Pe\dv\tau_q)|\Lm^{-1}\Pe\dv\tau_q).
\eeq

\beq\label{eq-57}
\nn&&\fr{d}{dt}(\Pe u_q|\Pe\dv\tau_q)+\fr{\om}{\We}\|\nb\Pe u_q\|_{L^2}^2-\fr{1}{\Rey}\|\Pe\dv\tau_q\|_{L^2}^2
-\fr{1-\om}{\Rey}(\Dl\Pe u_q|\Pe\dv\tau_q)\\
&&+\fr{1}{\We}(\Pe u_q|\Pe\dv\tau_q)=(g_q^\flat|\Pe\dv\tau_q)+(h_q^\flat|\Pe u_q)+\int\dv v_q\Pe u_q\cdot\Pe\dv\tau_q.
\eeq

\beq\label{eq-58}
&&\nn\fr{d}{dt}(\Pe u_q|\Dl \Pe \dv\tau_q)+\fr{1}{\Rey}\|\nb\Pe\dv\tau_q\|_{L^2}^2-\fr{\om}{\We}\|\Dl\Pe u_q\|_{L^2}^2\\
&&\nn-\fr{1-\om}{\Rey}(\Dl\Pe u_q|\Dl\Pe\dv\tau_q)+\fr{1}{\We}(\Dl \Pe u_q|\Pe\dv\tau_q)\\
&=&(g_q^\flat|\Dl\Pe\dv\tau_q)+(h_q^\flat|\Dl\Pe u_q)-(v_q\cdot\nb\Pe u_q|\Dl\Pe\dv\tau_q)-(v_q\cdot\nb\Pe\dv\tau_q|\Dl\Pe u_q).
\eeq
Denote
\beqno
\tl{Y}_q:=\left[2\|\Pe u_q\|_{L^2}^2+\left\|\fr{(1-\om)\We}{\om\Rey}\Pe\dv\tau_q\right\|_{L^2}^2-2\left(\Pe u_q|\fr{(1-\om)\We}{\om\Rey}\Pe\dv\tau_q\right)\right]^{\fr12}, \quad \mathrm{for}\ q>q_0,
\eeqno

\beqno
\tl{Y}_q:=\left[\|\Pe u_q\|_{L^2}^2+M'\|\fr{\We}{\Rey}\Pe\dv\tau_q\|_{L^2}^2+2(\Pe u_q|\fr{\We}{\Rey}\Pe\dv\tau_q)+2(\Pe u_q|\fr{\We}{\Rey}\Dl\Pe\dv\tau_q)\right]^{\fr12}, \quad \mathrm{for}\ q\le q_1,
\eeqno

\beqno
\tl{Y}_q:=\left[\|\Pe u_q\|_{L^2}^2+\fr{\We}{\om\Rey}\|\Lm^{-1}\Pe\dv\tau_q\|_{L^2}^2
-\fr{2\beta(1-\om)\We}{\om\Rey}(\Pe u_q|\Pe\dv\tau_q)\right]^{\fr12}, \quad \mathrm{for}\ q_1<q\le q_0,
\eeqno
where $q_0, q_1, M', \beta$ are given as in the compressible part. Next, we shall bound $\tl{Y}_q$ for all $q\in \Z$.
\bigbreak
\noindent{\em Step 1: high frequencies.}\par
\noindent Similar to \eqref{eq-2}, from \eqref{eq-55}, \eqref{eq-57}, we have
\beq\label{eq-59}
\nn&&\fr12\fr{d}{dt}\left(\left\|\fr{(1-\om)\We}{\om\Rey}\Pe\dv\tau_q\right\|_{L^2}^2-2\left(\Pe u_q|\fr{(1-\om)\We}{\om\Rey}\Pe\dv\tau_q\right)\right)\\
\nn&&+\fr{1}{(1-\om)\We}\left\|\fr{(1-\om)\We}{\om\Rey}\Pe\dv\tau_q\right\|^2_{L^2}-\fr{1-\om}{\Rey}\|\nb\Pe u_q\|_{L^2}^2-\fr{1}{\We}\left(\Pe u_q|\fr{(1-\om)\We}{\om\Rey}\Pe\dv\tau_q\right)\\
&=&\left(\fr{(1-\om)\We}{\om\Rey}\right)^2 {R.H.S.\ of\ } \eqref{eq-55}-\fr{(1-\om)\We}{\om\Rey}{R.H.S.\ of\ } \eqref{eq-57}.
\eeq
Multiplying \eqref{eq-54} by 2, and then adding the resulting inequality to \eqref{eq-59} yields
\beq\label{eq-60}
\nn&&\fr12\fr{d}{dt}\left(2\|\Pe u_q\|_{L^2}^2+\left\|\fr{(1-\om)\We}{\om\Rey}\Pe\dv\tau_q\right\|_{L^2}^2-2\left(\Pe u_q|\fr{(1-\om)\We}{\om\Rey}\Pe\dv\tau_q\right)\right)\\
\nn&&+\fr{1-\om}{\Rey}\|\nb\Pe u_q\|_{L^2}^2+\fr{1}{(1-\om)\We}\left\|\fr{(1-\om)\We}{\om\Rey}\Pe\dv\tau_q\right\|^2_{L^2}\\
\nn&&-\fr{2}{\Rey}(\Pe\dv\tau_q|\Pe u_q)-\fr{1}{\We}\left(\Pe u_q|\fr{(1-\om)\We}{\om\Rey}\Pe\dv\tau_q\right)\\
&=&2{R.H.S.\ of\ } \eqref{eq-54}+\left(\fr{(1-\om)\We}{\om\Rey}\right)^2 {R.H.S.\ of\ } \eqref{eq-55}-\fr{(1-\om)\We}{\om\Rey}{R.H.S.\ of\ } \eqref{eq-57}.
\eeq
Obviously,
\beq\label{eq-61}
\nn&&2\|\Pe u_q\|_{L^2}^2+\left\|\fr{(1-\om)\We}{\om\Rey}\Pe\dv\tau_q\right\|_{L^2}^2-2\left(\Pe u_q|\fr{(1-\om)\We}{\om\Rey}\Pe\dv\tau_q\right)\\
&\approx&\|\Pe u_q\|_{L^2}^2+\left\|\fr{(1-\om)\We}{\om\Rey}\Pe\dv\tau_q\right\|_{L^2}^2,
\eeq
and
\beq\label{eq-62}
\nn&&\left|\fr{2}{\Rey}(\Pe\dv\tau_q|\Pe u_q)+\fr{1}{\We}\left(\Pe u_q|\fr{(1-\om)\We}{\om\Rey}\Pe\dv\tau_q\right)\right|\\
\nn&=&\fr{1+\om}{1-\om}\fr{1}{\We}\left|\left(\Pe u_q|\fr{(1-\om)\We}{\om\Rey}\Pe\dv\tau_q\right)\right|\\
\nn&\le&\fr12\fr{1}{1-\om}\fr{1}{\We}\|\fr{(1-\om)\We}{\om\Rey}\Pe\dv\tau_q\|_{L^2}^2+\left(\fr432^{-q_0}\right)^2\fr{2}{(1-\om)\We}\|\nb \Pe u_q\|_{L^2}^2\\
&\le&\fr12\fr{1}{1-\om}\fr{1}{\We}\|\fr{(1-\om)\We}{\om\Rey}\Pe\dv\tau_q\|_{L^2}^2+\fr12\fr{1-\om}{\Rey}\|\nb\Pe u_q\|_{L^2}^2,
\eeq
provided
\be\label{eq-63}
\left(\fr432^{-q_0}\right)^2\fr{2}{(1-\om)\We}\le\fr12\fr{1-\om}{\Rey},\qquad\mathrm{i. e.}\qquad 2^{q_0}\ge \fr{8}{3(1-\om)}\sqrt{\fr{\Rey}{\We}}.
\ee
It is easy to see that
\beno
\left(\fr43\right)^{\fr32}\sqrt{\fr{2\Rey(\Rey\We+2)}{(1-\om)^3\We}}\ge\fr{8}{3(1-\om)}\sqrt{\fr{\Rey}{\We}}.
\eeno
Therefore, \eqref{eq-63} is a consequence of \eqref{eq-8+}.
It follows from \eqref{eq-62} and \eqref{eq-63} that
\beq\label{eq-64}
\nn&&\fr{1-\om}{\Rey}\|\nb\Pe u_q\|_{L^2}^2+\fr{1}{(1-\om)\We}\left\|\fr{(1-\om)\We}{\om\Rey}\Pe\dv\tau_q\right\|^2_{L^2}\\
\nn&&-\fr{2}{\Rey}(\Pe\dv\tau_q|\Pe u_q)-\fr{1}{\We}\left(\Pe u_q|\fr{(1-\om)\We}{\om\Rey}\Pe\dv\tau_q\right)\\
&\ge&\fr12\fr{1-\om}{\Rey}\|\nb\Pe u_q\|_{L^2}^2+\fr12\fr{1}{(1-\om)\We}\left\|\fr{(1-\om)\We}{\om\Rey}\Pe\dv\tau_q\right\|^2_{L^2}.
\eeq
Using Cauchy-Schwarz inequality and commutator estimates, we can bound the right hand side of \eqref{eq-60} as follows,
\beq\label{eq-65}
\nn&&R.H.S. of \eqref{eq-60}\\
\nn&\le&C\left(\|g_q^\flat\|_{L^2}+\|h_q^\flat\|_{L^2}\right)\left(\|\nb\Pe u_q\|_{L^2}+\left\|\fr{(1-\om)\We}{\om\Rey}\Pe\dv\tau_q\right\|_{L^2}\right)\\
\nn&&+C\|\nb v\|_{L^\infty}\left(\|\nb\Pe u_q\|_{L^2}+\left\|\fr{(1-\om)\We}{\om\Rey}\Pe\dv\tau_q\right\|_{L^2}\right)^2\\
&\le&C\left(\|\Pe G_q\|_{L^2}+\|\Pe H_q\|_{L^2}+\|\nb v\|_{L^\infty}\sum_{|q-q'|\le4}(\|u_{q'}\|_{L^2}+\|\dv\tau_{q'}\|_{L^2})\right).
\eeq
Substituting \eqref{eq-65} into \eqref{eq-60}, using \eqref{eq-61} and \eqref{eq-64}, we find that there exist constants $\tl{c}_1$ and $C$ depending on $d, \Rey, \We, $ and $\om$, such that if $q\ge q_0$, there holds
\beq\label{eq66}
\fr{d}{dt}\tl{Y}_q+\tl{c}_1\tl{Y}_q\le C\left(\|\Pe G_q\|_{L^2}+\|\Pe H_q\|_{L^2}+\|\nb v\|_{L^\infty}\sum_{|q-q'|\le4}(\|u_{q'}\|_{L^2}+\|\dv\tau_{q'}\|_{L^2})\right).
\eeq
{\em Step 2: low frequencies.}\par
Part {\bf(i)}. $q\le q_1$. Similar to \eqref{eq-10}, a linear combination of \eqref{eq-54}, \eqref{eq-55}, \eqref{eq-57} and \eqref{eq-58} yields
\beq\label{eq67}
\nn&&\fr12\fr{d}{dt}\left[\|\Pe u_q\|_{L^2}^2+M'\|\fr{\We}{\Rey}\Pe\dv\tau_q\|_{L^2}^2+2(\Pe u_q|\fr{\We}{\Rey}\Pe\dv\tau_q)+2(\Pe u_q|\fr{\We}{\Rey}\Dl\Pe\dv\tau_q)\right]\\
\nn&&+\fr{1}{\Rey}\|\nb\Pe u_q\|_{L^2}^2-\fr{\om}{\Rey}\|\Dl\Pe u_q\|_{L^2}^2+\fr{M'-1}{\We}\|\fr{\We}{\Rey}\Pe\dv\tau_q\|_{L^2}^2\\
\nn&&+\fr{1}{\We}\|\fr{\We}{\Rey}\nb\Pe\dv\tau_q\|^2_{L^2}-\fr{1-\om}{\Rey}(\Dl\Pe u_q|\fr{\We}{\Rey}\Dl\Pe\dv\tau_q)\\
\nn&&+\left(\fr{1}{\We}-\fr{(1-\om)+M'\om}{\Rey}\right)(\Dl\Pe u_q|\fr{\We}{\Rey}\Pe\dv\tau_q)\\
&=&R.H.S. \ of\ \eqref{eq-54}+M'\left(\fr{\We}{\Rey}\right)^2R.H.S. \ of\ \eqref{eq-55}+\fr{\We}{\Rey}R.H.S. \ of\ \eqref{eq-57}\ and\ \eqref{eq-58},
\eeq
where $M'$ is the same as in \eqref{eq-10}. The cross terms in \eqref{eq67} and the right hand side can be estimated  in a similar manner as those in \eqref{eq-10}, accordingly, we infer from \eqref{eq67} that there exist constants $\tl{c}_2$ and $C$ depending on $d, \Rey, \We, $ and $\om$, such that if $q\le q_1$, there holds
\be\label{eq68}
\fr{d}{dt}\tl{Y}_q+\tl{c}_22^{2q}\tl{Y}_q\le C\left(\|\Pe G_q\|_{L^2}+\|\Pe H_q\|_{L^2}+\|\nb v\|_{L^\infty}\sum_{|q-q'|\le4}(\|u_{q'}\|_{L^2}+\|\dv\tau_{q'}\|_{L^2})\right).
\ee

Part {\bf(ii)}. $q_1<q\le q_0$. Similar to \eqref{eq-32}, by a linear combination of \eqref{eq-54}, \eqref{eq-56} and \eqref{eq-57}, we get
\beq\label{eq69}
\nn&&\fr12\fr{d}{dt}\left(\|\Pe u_q\|_{L^2}^2+\fr{\We}{\om\Rey}\|\Lm^{-1}\Pe\dv\tau_q\|_{L^2}^2
-\fr{2\beta(1-\om)\We}{\om\Rey}(\Pe u_q|\Pe\dv\tau_q)\right)\\
\nn&&+(1-\beta)\fr{1-\om}{\Rey}\|\nb\Pe u_q\|_{L^2}^2+\fr{\beta(1-\om)\We}{\om(\Rey)^2}\|\Pe\dv\tau_q\|_{L^2}^2+\fr{1}{\om\Rey}\|\Lm^{-1}\Pe\dv\tau_q\|_{L^2}^2\\
\nn&&+\fr{\beta(1-\om)^2\We}{\om(\Rey)^2}(\Dl\Pe u_q|\Pe\dv\tau_q)-\fr{\beta(1-\om)}{\om\Rey}(\Pe u_q|\Pe\dv\tau_q)\\
&\le&R.H. S. \ of\ \eqref{eq-54}+\fr{\We}{\om\Rey}R. H. S. \ of\ \eqref{eq-56}-\fr{\beta(1-\om)\We}{\om\Rey}R. H. S. \ of\ \eqref{eq-57},
\eeq
where $\beta$ is the same as in \eqref{eq-32}. Arguing as in the corresponding compressible case, we find that there exist constants
$\tl{c}_3$ and $C$ depending on $d, \Rey, \We, $ and $\om$, such that  if $q_1<q\le q_0$, \eqref{eq69} implies
\be\label{eq70}
\fr{d}{dt}\tl{Y}_q+\tl{c}_32^{2q}\tl{Y}_q\le C\left(\|\Pe G_q\|_{L^2}+\|\Pe H_q\|_{L^2}+\|\nb v\|_{L^\infty}\sum_{|q-q'|\le4}(\|u_{q'}\|_{L^2}+\|\dv\tau_{q'}\|_{L^2})\right).
\ee

\bigbreak
\noindent{\bf Global estimates of $(a, u, \dv\tau)$}\par
\noindent Let $X_q:=Y_q+\tl{Y}_q$. Moreover, we set
\beno
s(q):=\begin{cases}
  2^{q},\quad\, \mathrm{if}\quad q>q_0,\\
  1,\ \ \,\quad\mathrm{if}\quad q\le q_0,
   \end{cases}
  \quad\mathrm{and}\quad
  \tl{s}(q):=\begin{cases}
   1,\ \ \,\quad\mathrm{if}\quad q>q_0,\\
  2^{2q},\quad\mathrm{if}\quad q\le q_0.
   \end{cases}
\eeno
Recalling the definition of $Y_q$ and $\tl{Y}_q$, using Bernstein's inequalities, we infer from  \eqref{eq-4}, \eqref{eq-16}, \eqref{eq-37} and the corresponding estimates for incompressible part that
\be\label{X=}
X_q\approx s(q)\| a_q\|_{L^2}+\|u_q\|_{L^2}+\|\dv\tau_q\|_{L^2}.
\ee
Now collecting \eqref{es-CH}, \eqref{es-CLL}, \eqref{es-CLH}, \eqref{eq66}, \eqref{eq68} and \eqref{eq70}, we conclude  that there exist constants $\bar{c}$ and $C$ depending on $d, \Rey, \We, $ and $\om$, such that for all $q\in \Z$
\beq\label{eq71}
\nn&&\fr{d}{dt}X_q+\bar{c}\tl{s}(q)X_q\\
&\le& C\left({s}(q)\|F_q\|_{L^2}+\|G_q\|_{L^2}+\|H_q\|_{L^2}+\|\nb v\|_{L^\infty}\sum_{|q-q'|\le4}X_{q'}\right).
\eeq
Performing a time integration in \eqref{eq71},  multiplying the resulting inequality by $2^{q(s-1)}$, and taking sum w. r. t. $q$ over $\Z$ yields
 \beqno
\nn&&\sum_{q\in\Z}2^{q(s-1)}X_q(t)+\bar{c}\sum_{q\in\Z}\tl{s}(q)2^{q(s-1)}\int_0^tX_qdt'\\
\nn&\le&\sum_{q\in\Z}2^{q(s-1)}X_q(0)+C(\|F\|_{L^1(\dot{B}^{s-1,s}_{2,1})}+\|G\|_{L^1_t(\dot{B}^{s-1}_{2,1})}
+\|H\|_{L^1_t(\dot{B}^{s-1}_{2,1})})\\
&&+C\int_0^t\|\nb v\|_{L^\infty}\sum_{q\in\Z}2^{q(s-1)}X_{q}dt'.
\eeqno
Using the fact \eqref{X=}, this inequality is nothing but
 \beqno
 \nn&&\|a(t)\|_{\dot{B}^{s-1,s}_{2,1}}+\|u(t)\|_{\dot{B}^{s-1}_{2,1}}+\|\dv\tau(t)\|_{\dot{B}^{s-1}_{2,1}}\\
 \nn&&+\|a\|_{L^1_t(\dot{B}^{s+1,s}_{2,1})}+\|u\|_{L^1_t(\dot{B}^{s+1,s-1}_{2,1})}+\|\dv\tau\|_{L^1_t(\dot{B}^{s+1,s-1}_{2,1})}\\                                                           \nn&\le&C\left(\|a_0\|_{\dot{B}^{s-1,s}_{2,1}}+\|u_0\|_{\dot{B}^{s-1}_{2,1}}+\|\dv\tau_0\|_{\dot{B}^{s-1}_{2,1}}+\|F\|_{L^1(\dot{B}^{s-1,s}_{2,1})}
 +\|G\|_{L^1_t(\dot{B}^{s-1}_{2,1})}+\|H\|_{L^1_t(\dot{B}^{s-1}_{2,1})} \right.\\
 &&\left. +C\int_0^t\|\nb v\|_{L^\infty}(\|a(t')\|_{\dot{B}^{s-1,s}_{2,1}}+\|u(t')\|_{\dot{B}^{s-1}_{2,1}}+\|\dv\tau(t')\|_{\dot{B}^{s-1}_{2,1}})dt'\right),
 \eeqno
Gronwall's inequality implies then that
\beq\label{eq74}
\nn&&\|a(t)\|_{\dot{B}^{s-1,s}_{2,1}}+\|u(t)\|_{\dot{B}^{s-1}_{2,1}}+\|\dv\tau(t)\|_{\dot{B}^{s-1}_{2,1}}\\
\nn&&+\|a\|_{L^1_t(\dot{B}^{s+1,s}_{2,1})}+\|u\|_{L^1_t(\dot{B}^{s+1,s-1}_{2,1})}+\|\dv\tau\|_{L^1_t(\dot{B}^{s+1,s-1}_{2,1})}\\
\nn&\le&C\exp\left(C\|\nb v\|_{L^1_t(L^\infty)}\right)\left(\|a_0\|_{\dot{B}^{s-1,s}_{2,1}}+\|u_0\|_{\dot{B}^{s-1}_{2,1}}+\|\dv\tau_0\|_{\dot{B}^{s-1}_{2,1}}\right.\\
&&\left.+\|F\|_{L^1(\dot{B}^{s-1,s}_{2,1})}+\|G\|_{L^1_t(\dot{B}^{s-1}_{2,1})}
+\|H\|_{L^1_t(\dot{B}^{s-1}_{2,1})}\right).
\eeq
\bigbreak
\noindent{\bf The smoothing effect of $u$}\par
\noindent Applying $\ddl$ to the equation of $\eqref{piOB}_2$ yields
\beqno
\pr_t u_q+v_q\cdot\nb u_q-\fr{1}{\Rey}\mathcal{A}u_q+\nb a_q-\fr{1}{\Rey}\dv\tau_q=G_q+(v_q\cdot\nb u_q-\ddl\dot{T}_v\cdot\nb u).
\eeqno
Taking the inner product of the above equation with $u_q$, integrating by parts, we have
\beqno
&&\fr12\fr{d}{dt}\|u_q\|_{L^2}^2+\fr{1-\om}{\Rey}\left(\|\nb u_q\|_{L^2}^2+\|\dv u_q\|_{L^2}^2\right)\\
&=&\fr12\int\dv v_q|u_q|^2-(\nb a_q|u_q)+\fr{1}{\Rey}(\dv\tau_q|u_q)+(G_q|u_q)+(v_q\cdot\nb u_q-\ddl\dot{T}_v\cdot\nb u|u_q).
\eeqno
By virtue of Bernstein's inequality and  the commutator estimates, we easily get
\beqno
&&\|u_q(t)\|_{L^2}+\kappa\fr{1-\om}{\Rey}2^{2q}\|u_q\|_{L^1_t(L^2)}\\
&\le&\|u_q(0)\|_{L^2}+C2^q\| a_q\|_{L^1_t(L^2)}+\fr{1}{\Rey}\|\dv\tau_q\|_{L^1_t(L^2)}+\|G_q\|_{L^1_t(L^2)}+C\int_0^t\|\nb v_q\|_{L^\infty}\sum_{|q'-q|\le 4}\|u_{q'}\|_{L^2}dt',
\eeqno
with positive constant $\kappa$ depending on $d$. Multiplying this inequality by $2^{q(s-1)}$ and sum over $q>q_0$, we arrive at
\beqno
&&\|u^h(t)\|_{\dot{B}^{s-1}_{2,1}}+\kappa\|u^h(t)\|_{L^1_t(\dot{B}^{s+1}_{2,1})}\\
&\le&\|u^h_0\|_{\dot{B}^{s-1}_{2,1}}+C\left(\|a^h\|_{L^1_t(\dot{B}^s_{2,1})}
+\|\dv\tau^h\|_{L^1_t(\dot{B}^{s-1}_{2,1})}+\|G\|_{L^1_t(\dot{B}^{s-1}_{2,1})}\right.\\
&&\left.+C\|\nb v\|_{L^1_t(L^\infty)}\|u^l\|_{L^\infty_t(\dot{B}^{s-1}_{2,1})}\right)+C\int_0^t\|\nb v\|_{L^\infty}\|u^h\|_{\dot{B}^{s-1}_{2,1}}dt'.
\eeqno
Consequently, using Gronwall's inequality, thanks to \eqref{eq74}, we are led to
\beq\label{eq75}
\nn&&\|a(t)\|_{\dot{B}^{s-1,s}_{2,1}}+\|u(t)\|_{\dot{B}^{s-1}_{2,1}}+\|\dv\tau(t)\|_{\dot{B}^{s-1}_{2,1}}\\
\nn&&+\|a\|_{L^1_t(\dot{B}^{s+1,s}_{2,1})}+\|u\|_{L^1_t(\dot{B}^{s+1}_{2,1})}+\|\dv\tau\|_{L^1_t(\dot{B}^{s+1,s-1}_{2,1})}\\
\nn&\le&C\exp\left(C\|\nb v\|_{L^1_t(L^\infty)}\right)\left(\|a_0\|_{\dot{B}^{s-1,s}_{2,1}}+\|u_0\|_{\dot{B}^{s-1}_{2,1}}+\|\dv\tau_0\|_{\dot{B}^{s-1}_{2,1}}\right.\\
&&\left.+\|F\|_{L^1(\dot{B}^{s-1,s}_{2,1})}+\|G\|_{L^1_t(\dot{B}^{s-1}_{2,1})}
+\|H\|_{L^1_t(\dot{B}^{s-1}_{2,1})}\right).
\eeq

\bigbreak
\noindent{\bf The damping effect of $\tau$}\par
\noindent Applying $\ddl$ to the equation of $\eqref{piOB}_3$ yields
\beqno
\pr_t \tau_q+v_q\cdot\nb \tau_q+\fr{\tau_q}{\We}-\fr{2\om}{\We}D(u_q)=L_q+(v_q\cdot\nb \tau_q-\ddl\dot{T}_v\cdot\nb \tau).
\eeqno
Taking the inner product of the above equation with $\tau_q$, integrating by parts, we have
\beqno
&&\fr12\fr{d}{dt}\|\tau_q\|_{L^2}^2+\fr{1}{\We}\|\tau_q\|_{L^2}^2\\
&=&\fr12\int\dv v_q|\tau_q|^2+\fr{2\om}{\We}(D(u_q)|\tau_q)+(L_q|u_q)+(v_q\cdot\nb \tau_q-\ddl\dot{T}_v\cdot\nb \tau|\tau_q).
\eeqno
Similar to the estimates of $u^h$, one easily deduces that
\beqno
\|\tau_q(t)\|_{L^2}+\fr{1}{\We}\|\tau_q\|_{L^1_t(L^2)}\le C\left(\|\tau_q(0)\|_{L^2}+2^q\|u_q\|_{L^2}+\|L_q\|_{L^2}+\int_0^t\|\nb v\|_{L^\infty}\sum_{|q-q'|\le4}\|\tau_{q'}\|_{L^2}dt'\right).
\eeqno
Multiplying the above equation by $2^{qs}$, and summing over $q\in \Z$, we find that
\beqno
\|\tau(t)\|_{\dot{B}^{s}_{2,1}}+\fr{1}{\We}\|\tau\|_{L^1_t(\dot{B}^{s}_{2,1})}\le C\left(\|\tau_0\|_{\dot{B}^s_{2,1}}+\|u\|_{L^1_t(\dot{B}^{s+1}_{2,1})}+\|L\|_{L^1_t(\dot{B}^s_{2,1})}+\int_0^t\|\nb v\|_{L^\infty}\|\tau\|_{\dot{B}^s_{2,1}}dt'\right).
\eeqno
Gronwall's inequality implies then that
\beq\label{eq76}
\|\tau(t)\|_{\dot{B}^{s}_{2,1}}+\fr{1}{\We}\|\tau\|_{L^1_t(\dot{B}^{s}_{2,1})}\le C\exp\left(C\|\nb v\|_{L^1_t(L^\infty)}\right)\left(\|\tau_0\|_{\dot{B}^s_{2,1}}+\|u\|_{L^1_t(\dot{B}^{s+1}_{2,1})}+\|L\|_{L^1_t(\dot{B}^s_{2,1})}\right).
\eeq
To close the estimate, substituting \eqref{eq75} into \eqref{eq76}, we conclude that
\beq\label{eq77}
\nn&&\|a(t)\|_{\dot{B}^{s-1,s}_{2,1}}+\|u(t)\|_{\dot{B}^{s-1}_{2,1}}+\|\tau(t)\|_{\dot{B}^{s}_{2,1}}\\
\nn&&+\|a\|_{L^1_t(\dot{B}^{s+1,s}_{2,1})}+\|u\|_{L^1_t(\dot{B}^{s+1}_{2,1})}+\|\tau\|_{L^1_t(\dot{B}^{s}_{2,1})}+\|\dv\tau\|_{L^1_t(\dot{B}^{s+1,s-1}_{2,1})}\\
\nn&\le&C\exp\left(C\|\nb v\|_{L^1_t(L^\infty)}\right)\left(\|a_0\|_{\dot{B}^{s-1,s}_{2,1}}+\|u_0\|_{\dot{B}^{s-1}_{2,1}}+\|\tau_0\|_{\dot{B}^{s}_{2,1}}\right.\\
&&\left.+\|F\|_{L^1(\dot{B}^{s-1,s}_{2,1})}+\|G\|_{L^1_t(\dot{B}^{s-1}_{2,1})}
+\|H\|_{L^1_t(\dot{B}^{s-1}_{2,1})}+\|L\|_{L^1_t(\dot{B}^{s}_{2,1})}\right).
\eeq
Recalling that $H^k:=(\dv L)^k-\sum_{1\le i, j\le d}\dot{T}_{\pr_iv^j}\pr_j\tau^{i,k}$, by virtue of Lemma \ref{Bernstein} and Proposition \ref{p-TR}, it is easy to see that
\be\label{eq77+}
\|H\|_{L^1_t(\dot{B}^{s-1}_{2,1})}\le C\|L\|_{L^1_t(\dot{B}^{s}_{2,1})}+C\int_0^t\|\nb v\|_{L^\infty}\|\tau\|_{\dot{B}^{s}_{2,1}}dt'.
\ee
Combining \eqref{eq77} with \eqref{eq77+}, then \eqref{eq-prop} follows immediately. This completes the proof of Proposition \ref{piOB}.
\end{proof}
\section{Proof of Theorem\ref{thm-g}}
\noindent
Since the local existence and uniqueness of the solution $(a, u, \tau)$ to \eqref{COB2dimensionless} have been proved in \cite{FZ14}, it suffices to show $T^*=\infty$, where $T^*$ is the maximal existence time of $(a, u, \tau)$. To this end, let us now denote $V(t):=\int_0^t\|\nb u(t')\|_{L^\infty}dt'$,
\beno
X(t):=\|a\|_{L^\infty_t(\dot{B}^{\fr{d}{2}-1,\fr{d}{2}}_{2,1})\cap L^1_t(\dot{B}^{\fr{d}{2}+1,\fr{d}{2}}_{2,1})}+\|u\|_{L^\infty_t(\dot{B}^{\fr{d}{2}-1}_{2,1})\cap L^1_t(\dot{B}^{\fr{d}{2}+1}_{2,1})}+\|\tau\|_{L^\infty_t(\dot{B}^{\fr{d}{2}}_{2,1})\cap L^1_t(\dot{B}^{\fr{d}{2}}_{2,1})},
\eeno
\beno
U(t):=\|a\|_{L^1_t(\dot{B}^{\fr{d}{2}+1,\fr{d}{2}}_{2,1})}+\|u\|_{L^1_t(\dot{B}^{\fr{d}{2}+1}_{2,1})}+\|\tau\|_{L^1_t(\dot{B}^{\fr{d}{2}}_{2,1})},
\eeno
and
\beno
X_0:=\|a_0\|_{\dot{B}^{\fr{d}{2}-1,\fr{d}{2}}_{2,1}}+\|u_0\|_{\dot{B}^{\fr{d}{2}-1}_{2,1}}+\|\tau_0\|_{\dot{B}^{\fr{d}{2}}_{2,1}}.
\eeno
Clearly, $U(t)$ is continuous with respect to time $t$.
Applying Proposition \ref{prop-global} with $s=\fr{d}{2}, v=u$ and
\begin{gather*}
F=-\dv(\dot{T}'_au),\\
G=-\fr{1}{\Rey}I(a)\left(\mathcal{A}u+\dv\tau\right)+\tl{K}(a)\nb a-\sum_{1\le j\le d}\dot{T}'_{\pr_j}u^j,\\
L=-g_\al(\tau, \nb u)-\sum_{1\le j\le d}\dot{T}'_{\pr_j\tau}u^j,
\end{gather*}
 and employing the product estimates in Besov spaces, it is not difficult to verity that
\beno
\|F\|_{L^1_t(\dot{B}^{\fr{d}{2}-1,\fr{d}{2}}_{2,1})}\le C\|a\|_{L^\infty_t(\dot{B}^{\fr{d}{2}-1,\fr{d}{2}}_{2,1})}\|u\|_{L^1_t(\dot{B}^{\fr{d}{2}+1}_{2,1})},
\eeno
\beqno
\|G\|_{L^1_t(\dot{B}^{\fr{d}{2}-1}_{2,1})}&\le&C\|a\|_{L^\infty_t(\dot{B}^{\fr{d}{2}}_{2,1})}\left(\|u\|_{L^1_t(\dot{B}^{\fr{d}{2}+1}_{2,1})}+\|\tau\|_{L^1_t(\dot{B}^{\fr{d}{2}}_{2,1})}\right)\\
&&+C\|a\|_{L^2_t(\dot{B}^{\fr{d}{2}}_{2,1})}^2+C\|u\|_{L^\infty_t(\dot{B}^{\fr{d}{2}-1}_{2,1})}\|u\|_{L^1_t(\dot{B}^{\fr{d}{2}+1}_{2,1})}\\
&\le&C\|a\|_{L^\infty_t(\dot{B}^{\fr{d}{2}-1,\fr{d}{2}}_{2,1})}\left(\|u\|_{L^1_t(\dot{B}^{\fr{d}{2}+1}_{2,1})}+\|\tau\|_{L^1_t(\dot{B}^{\fr{d}{2}}_{2,1})}\right)\\
&&+C\|a\|_{L^\infty_t(\dot{B}^{\fr{d}{2}-1,\fr{d}{2}}_{2,1})}\|a\|_{L^1_t(\dot{B}^{\fr{d}{2}+1,\fr{d}{2}}_{2,1})}
+C\|u\|_{L^\infty_t(\dot{B}^{\fr{d}{2}-1}_{2,1})}\|u\|_{L^1_t(\dot{B}^{\fr{d}{2}+1}_{2,1})},
\eeqno
and
\beqno
\|L\|_{L^1_t(\dot{B}^{\fr{d}{2}}_{2,1})}&\le& C\|\tau\|_{L^\infty_t(\dot{B}^{\fr{d}{2}}_{2,1})}\|\nb u\|_{L^1_t(\dot{B}^{\fr{d}{2}}_{2,1})}+\sum_{1\le j\le d}\|\dot{T}'_{\pr_j\tau}u^j\|_{L^1_t(\dot{B}^\fr{d}{2}_{2,1})}\\
&\le&C\|\tau\|_{L^\infty_t(\dot{B}^{\fr{d}{2}}_{2,1})}\|u\|_{L^1_t(\dot{B}^{\fr{d}{2}+1}_{2,1})}+C\|\nb\tau\|_{L^\infty_t(\dot{B}^{-1}_{\infty,\infty})}\|u\|_{L^1_t(\dot{B}^{\fr{d}{2}+1}_{2,1})}\\
&\le&C\|\tau\|_{L^\infty_t(\dot{B}^{\fr{d}{2}}_{2,1})}\|u\|_{L^1_t(\dot{B}^{\fr{d}{2}+1}_{2,1})}.
\eeqno
Collecting all these estimates, we infer from \eqref{eq-prop} that
\be\label{eq78}
X(t)\le Ce^{CV(t)}\left(X_0+X(t)U(t)\right).
\ee
Let us denote by $C_0$ the embedding constant of $\dot{B}^{\fr{d}{2}}_{2,1}\hookrightarrow L^\infty$. Then it is easy to see that $V(t)\le C_0U(t)$. Choosing a positive constant $c_\star$ satisfying
\be\label{eq79}
e^{CC_0c_\star}\le2,\quad \mathrm{and} \quad 2Cc_\star\le\fr12.
\ee
Define $T_1$ be the supremum of all time $T'\in[0,T^*)$ such that
\be\label{eq80}
U(t)\le c_\star, \quad\mathrm{for\ all}\quad t\in[0,T'].
\ee
Combining \eqref{eq79} with \eqref{eq80}, \eqref{eq78} reduces to
\be\label{eq81}
X(t)\le4CX_0, \quad \mathrm{for\ all}\quad t\in [0,T_1).
\ee
Noting that $U(t)\le X(t)$, for all $t\in[0,T^*)$. Taking $X_0$ so small that $4CX_0\le\fr{c_\star}{2}$, we then have
\be\label{eq82}
U(t)\le\fr{c_\star}{2}, \quad \mathrm{for\ all}\quad t\in[0,T_1).
\ee
This implies that $T_1=T^*$, and \eqref{eq80} holds on the interval $[0,T^*)$ provided $X_0\le\fr{c_\star}{8C}$. Accordingly, \eqref{eq81} holds with $T_1$ replaced by $T^*$, and hence $T^*=\infty$. This completes the proof of Theorem\ref{thm-g}.$\ \ \ \ \Box$

\bigbreak

\noindent{\bf Acknowledgments}
\bigbreak
Research supported by NSFC 11401237, 11271322, and 11331005. Part of this work was carried out while the  author is visiting the Department of Mathematics at the Technical University of Darmstadt.
I would express my gratitude to Prof. Matthias Hieber for his kind hospitality.
The author also would like to thank  Prof. Ting Zhang for his insightful suggestions.

\end{document}